\documentclass[aop,MSNbibl,seceqn,dvips]{arximspdf}

\doi{10.1214/12-AOP796} 
\volume{42}
\issue{2}
\pubyear{2014}
\firstpage{649}
\lastpage{688}

\makeatletter

\newtheorem{theorem}{Theorem}[section]
\newproclaim{assumption}[theorem]{Assumption}
\newproclaim{remark}[theorem]{Remark}
\newtheorem{lemma}[theorem]{Lemma}
\newtheorem{corollary}[theorem]{Corollary}
\newtheorem{proposition}[theorem]{Proposition}

\newcommand{\cB}{{\mathcal B}}
\newcommand{\cF}{{\mathcal F}}
\newcommand{\cG}{{\mathcal G}}
\newcommand{\cH}{{\mathcal H}}
\newcommand{\cP}{{\mathcal P}}
\newcommand{\cX}{{\mathcal X}}

\newcommand{\te}{{\theta}}
\newcommand{\Om}{{\Omega}}
\newcommand{\ve}{{\varepsilon}}
\newcommand{\del}{{\delta}}
\newcommand{\Gam}{{\Gamma}}
\newcommand{\sig}{{\sigma}}
\newcommand{\al}{{\alpha}}
\newcommand{\be}{{\beta}}
\newcommand{\ka}{{\kappa}}
\newcommand{\la}{{\lambda}}
\newcommand{\Up}{{\Upsilon}}
\newcommand{\vp}{{\varpi}}
\newcommand{\io}{{\iota}}
\newcommand{\vs}{{\varsigma}}

\newcommand{\bbR}{{\mathbb R}}
\newcommand{\bbZ}{{\mathbb Z}}

\makeatother

\begin{document}
\begin{frontmatter}

\title{Nonconventional limit theorems in discrete and continuous time
via martingales}
\runtitle{Nonconventional limit theorems}

\begin{aug}
\author[A]{\fnms{Yuri} \snm{Kifer}\corref{}\thanksref{t1}\ead[label=e1]{kifer@math.huji.ac.il}}
\and
\author[B]{\fnms{S. R. S.} \snm{Varadhan}\thanksref{t2}\ead[label=e2]{varadhan@cims.nyu.edu}}
\runauthor{Y. Kifer and S. R. S. Varadhan}
\affiliation{Hebrew University and New York University}
\address[A]{Institute of Mathematics\\
Hebrew University\\
Jerusalem 91904\\
Israel\\
\printead{e1}} 
\address[B]{Courant Institute\\
\quad for Mathematical Studies\\
New York University\\
251 Mercer St\\
New York, New York 10012\\
USA\\
\printead{e2}}
\end{aug}

\thankstext{t1}{Supported in part by ISF Grants 130/06 and 82/10.}

\thankstext{t2}{Supported in part by NSF Grants OISE 0730136 and DMS-09-04701.}

\received{\smonth{1} \syear{2012}}
\revised{\smonth{6} \syear{2012}}

%
\begin{abstract}
We obtain functional central limit theorems for both discrete time
expressions of the form $1/\sqrt{N}\sum_{n=1}^{[Nt]} (F
(X(q_1(n)),\ldots,X(q_\ell (n)) )- \bar F )$ and similar expressions in
the continuous time where the sum is replaced by an integral. Here
$X(n),n\geq0$ is a sufficiently fast mixing vector process with some
moment conditions and stationarity properties, $F$ is a continuous
function with polynomial growth and certain regularity properties,
$\bar F=\int F\,d(\mu\times\cdots\times\mu)$, $\mu$ is the distribution
of $X(0)$ and $q_i(n)=in$ for $i\le k\leq\ell$ while for $i>k$ they are
positive functions taking on integer values on integers with some
growth conditions which are satisfied, for instance, when $q_i$'s are
polynomials of increasing degrees. These results decisively generalize
[\textit{Probab. Theory Related Fields} \textbf{148} (2010) 71--106],
whose method was only applicable to the case $k=2$ under substantially
more restrictive moment and mixing conditions and which could not be
extended to convergence of processes and to the corresponding
continuous time case. As in [\textit{Probab. Theory Related Fields}
\textbf{148} (2010) 71--106], our results hold true when
$X_i(n)=T^nf_i$, where $T$ is a mixing subshift of finite type, a
hyperbolic diffeomorphism or an expanding transformation taken with a
Gibbs invariant measure, as well as in the case when
$X_i(n)=f_i(\Up_n)$, where $\Up_n$ is a Markov chain satisfying the
Doeblin condition considered as a stationary process with respect to
its invariant measure. Moreover, our relaxed mixing conditions yield
applications to other types of dynamical systems and Markov processes,
for instance, where a spectral gap can be established. The continuous
time version holds true when, for instance, $X_i(t)=f_i(\xi_t)$, where
$\xi_t$ is a nondegenerate continuous time Markov chain with a finite
state space or a nondegenerate diffusion on a compact manifold. A
partial motivation for such limit theorems is due to a series of papers
dealing with nonconventional ergodic averages.
\end{abstract}

%

\end{frontmatter}

\section{Introduction}\label{sec1}
Nonconventional ergodic theorems, known also after \cite{Be} as polynomial
ergodic theorems, studied the limits of expressions having the form
(cf. \cite{Fur}) $1/N\sum_{n=1}^NT^{q_1(n)}f_1\cdots
T^{q_\ell(n)}f_\ell$, where $T$ is a weakly mixing measure
preserving transformation, $f_i$'s are bounded measurable functions
and $q_i$'s are polynomials taking on integer values on the integers.
Originally, these results were motivated by applications to multiple recurrence
for dynamical systems, the functions $f_i$ being indicators of some measurable
sets.

After an ergodic theorem (or in the probabilistic language: the law of large
numbers) is established, it is natural to inquire whether a corresponding
central limit theorem holds true as well, though as usual under stronger
conditions.
In this paper we prove the functional central limit theorem (invariance
principle) for expressions of the form
%
\begin{equation}
\label{11} \frac{1}{\sqrt{N}}\sum_{n=1}^{[Nt]}
\bigl(F \bigl(X\bigl(q_1(n)\bigr),\ldots, X\bigl(q_\ell(n)
\bigr) \bigr) -\bar F \bigr)
\end{equation}
and for the corresponding continuous time expressions of the form
%
\begin{equation}
\label{12} \frac{1}{\sqrt N}\int_0^{[Nt]}
\bigl(F \bigl( X\bigl(q_1(t)\bigr),\ldots, X\bigl(q_\ell (t)
\bigr) \bigr)- \bar F \bigr)\,dt,
\end{equation}
where $\{X(n),n\geq0\}$, [or $\{X(t)\}, t\ge0$] is a sufficiently
fast mixing vector valued process with some stationarity properties satisfying certain
moment conditions, $F$~is a continuous function with polynomial growth and certain regularity
properties, $\bar F=\int F \,d(\mu\times\cdots\times\mu)$ where
$\mu$ is the common distribution of $X(n)$, $\{q_j(t)\}$ are positive functions
taking on integer values on integers in the discrete time case with $q_j(t)=jt$
for $j\le k$ and for $j>k$ they
satisfy certain growth conditions.
For instance, it would be enough if
$\{q_j (t)\}$ are polynomials of increasing degrees, though we actually
do not
need any polynomial structure of functions $q_j,  j>k$ which was
crucial in
papers dealing with nonconventional ergodic theorems cited above.

Our methods rely on a martingale approximations approach which has
played a decisive role in
most proofs of the central limit theorem during the last 50 years. In
view of
strong dependence on the future of summands in (\ref{11}),
application of
martingales in our setup does not seem plausible on first sight. It
turns out, somewhat surprisingly, that an appropriately modified martingale
approach still works well in our situation if we construct the filtration
of $\sig$-algebras so that in some sense ``future becomes present.'' Once
martingale approximations are constructed, it remains only to check
convergence of covariances which we do in Section \ref{sec4}, while
the whole
approach is explained and completed in Section \ref{sec5}.

Unlike
the classical situation, our functional central limit theorem yields a process
which has Gaussian distributions but not necessarily independent increments
and we demonstrate an explicit example of such limiting process with dependent
increments. This interesting effect rarely appears in natural models.
We obtain also a functional central limit theorem
in the corresponding continuous time case which only recently was treated
in the sense of nonconventional ergodic theorems (see~\cite{BLM}). It turns
out that the limiting process in the continuous time case has a somewhat
different structure than in the discrete time setup.
These results generalize \cite{Ki2}, where the partition
into blocks and the direct use of characteristic functions showed applicability
only to the case $k=2$ under more restrictive conditions and neither the
functional central limit theorem nor the continuous time case could be dealt
with by the method employed there.

Our results can be applied to large classes of stochastic processes
$X(n),
n\geq0$, in particular, to functions of Markov chains satisfying Doeblin's
condition or to those which are constructed from sufficiently fast mixing
dynamical systems.
The continuous time version holds true, in particular, when
$X(t)$ is a function of an irreducible continuous time Markov
chain or of a nondegenerate diffusion on a compact manifold or of
Ornstein--Uhlenbeck type processes.

\section{Preliminaries and main results}\label{sec2}

Our discrete time setup consists of a $\wp$-dimensional stochastic process
$\{X(n),  n=0,1,\ldots\}$ on a probability space $(\Om,\cF,P)$ and of a
family of
$\sig$-algebras $\cF_{kl}\subset\cF,  -\infty\leq k\leq
l\leq\infty$ such that $\cF_{kl}\subset\cF_{k'l'}$ if $k'\leq k$ and
$l'\geq l$. It is often convenient to measure the dependence between
two sub-$\sig$-algebras $\cG,\cH\subset\cF$ via the quantities
%
\begin{eqnarray}
\label{eq21}
&&\varpi_{q,p}(\cG,\cH)\nonumber\\[-8pt]\\[-8pt]
&&\qquad=\sup\bigl\{\bigl\| E [g|\cG ]-E[g]
\bigr\|_p\dvtx   g   \mbox{ is } \cH\mbox{-measurable and }
\|g\|_q\leq1\bigr\},\nonumber
\end{eqnarray}
where the supremum is taken over real functions and $\|\cdot\|_r$ is the
$L^r(\Om,\cF,P)$-norm. Then more familiar $\al,\rho,\phi$ and
$\psi$-mixing
(dependence) coefficients can be expressed via the formulas (see \cite{Bra},
Chapter 4)
\begin{eqnarray*}
\al(\cG,\cH)&=&\tfrac14\varpi_{\infty,1}(\cG,\cH), \qquad \rho(\cG,\cH)=
\varpi_{2,2} (\cG,\cH),
\\
\phi(\cG,\cH)&=&\tfrac12\varpi_{\infty,\infty}(\cG,\cH)
\quad\mbox{and}\quad   \psi(
\cG,\cH)=\varpi_{1,\infty}(\cG,\cH).
\end{eqnarray*}
We set also
%
\begin{equation}
\label{eq22} \varpi_{q,p}(n)=\sup_{k\geq0}\varpi_{q,p}(
\cF_{-\infty,k},\cF_{k+n,\infty})
\end{equation}
and, accordingly,
\begin{eqnarray*}
\al(n)&=&\tfrac{1}{4}\varpi_{\infty,1}(n),\qquad \rho(n)=
\varpi_{2,2}(n),\\
\phi(n)&=&\tfrac12\varpi_{\infty,\infty}(n), \qquad \psi(n)=
\varpi_{1,\infty}(n).
\end{eqnarray*}
We will impose mixing rates, that is, rates of decay of $\varpi_{q,p}(n)$
requiring that
%
\begin{equation}
\label{eq23} C(q,p)=\sum_{n\geq1}
\varpi_{q,p}(n)
\end{equation}
is finite for some choices of $p$ and $q$.
Our setup includes also conditions on the approximation rate
%
\begin{equation}
\label{20} \beta(p,r)=\sup_{k\geq0}\bigl\|X(k)-E \bigl[X(k)|
\cF_{k-r,k+r} \bigr]\bigr\|_p.
\end{equation}
In what follows we can always extend the definitions of $\cF_{kl}$
given only for $k,l\geq0$ to negative $k$ by defining $\cF_{kl}=\cF_{0l}$
for $k<0$ and $l\geq0$. Furthermore, we do not require stationarity of the
process $X(n), n\geq0$, assuming only that the distribution of $X(n)$
does not
depend on $n$ and the joint distribution of $\{X(n), X(n')\}$ depends
only on
$n-n'$ which we write for further references by
%
\begin{equation}
\label{21} X(n)\stackrel{d} {\sim}\mu
\quad\mbox{and}\quad\bigl(X(n),X
\bigl(n'\bigr)\bigr)\stackrel{d} {\sim}\mu_{n-n'}
\qquad\mbox{for all }  n,n',
\end{equation}
where $Y\stackrel{d}{\sim}\mu$ means that $Y$ has $\mu$ for its
distribution.

Next, let $F= F(x_1,\ldots,x_\ell),  x_j\in\bbR^{\wp}$ be a function on
$\bbR^{\wp\ell}$ such that for some $\iota,K>0,\ka\in(0,1]$ and all
$x_i,y_i\in\bbR^{\wp}, i=1,\ldots,\ell$, we have
%
\begin{eqnarray}
\label{24}
&&\bigl|F(x_1,\ldots,x_\ell)-F(y_1,\ldots,y_\ell)\bigr|\nonumber\\[-8pt]\\[-8pt]
&&\qquad\leq K \Biggl[1+\sum^\ell_{j=1}|x_j|^\iota+
\sum^\ell_{j=1} |y_j|^\iota
\Biggr]\sum^\ell_{j=1}|x_j-y_j|^\ka\nonumber
\end{eqnarray}
and
%
\begin{equation}
\label{24b} \bigl|F(x_1,\ldots,x_\ell)\bigr|\leq K \Biggl[ 1+\sum
^\ell_{j=1}|x_j|^{\iota}
\Biggr].
\end{equation}

To simplify formulas, we assume a centering condition
%
\begin{equation}
\label{24c} \bar F=\int F(x_1,\ldots,x_\ell) \,d
\mu(x_1)\cdots d\mu(x_\ell)=0,
\end{equation}
which is not really a restriction since we can always replace $F$ by
$F-\bar F$. Our goal is to prove a functional central limit theorem for
%
\begin{equation}
\label{24d} \xi_N(t)=\frac{1}{\sqrt N}\sum
_{n=1}^{[Nt]}F \bigl(X\bigl(q_1(n)
\bigr),\ldots,X\bigl(q_\ell(n)\bigr) \bigr)
\quad\mbox{and}\quad t\in[0,T],
\end{equation}
where $q_1(n)< q_2(n) <\cdots< q_\ell(n)$ are increasing functions taking
on integer values on integers and such that for $j\leq k$, $q_j(n)=jn$,
whereas the remaining ones grow faster in $n$. We assume that
for $k+1\le i \le\ell$,
%
\begin{equation}
\label{2q1} \lim_{n\to\infty}\bigl(q_i(n+1)-q_i(n)
\bigr)=\infty
\end{equation}
and for $i\ge k$ and any $\epsilon> 0 $,
%
\begin{equation}
\label{2q2} \liminf_{n\to\infty}\bigl(q_{i+1}(\epsilon
n)-q_i(n)\bigr)>0,
\end{equation}
which implies because of (\ref{2q1}) that
%
\begin{equation}
\label{2q3} \lim_{n\to\infty}\bigl(q_{i+1}(\epsilon
n)-q_i(n)\bigr)=\infty.
\end{equation}
To shorten some of the arguments, we assumed that $q_i(n)$ is
increasing in both
$n$ and $i$ but, in fact, (\ref{2q1}) and (\ref{2q2}) imply already that
this holds true for all $n$ large enough, which suffices for our purposes.
For each $\theta>0$ set
%
\begin{equation}
\label{2m1} \gamma_\theta^\theta= \|X\|_\theta^\theta=
E\bigl|X(n)\bigr|^\theta = \int|x|^\theta \,d\mu.
\end{equation}
Our main result relies on the following.
%
\begin{assumption}\label{Hyp2} With $d=(\ell-1)\wp$ there exist
$\infty>
p,q\ge1$ and $\delta,m >0$ with $ \delta< \ka-\frac dp$ satisfying
%
\begin{eqnarray}
\label{eqsix1} &\displaystyle \sum_{n=0}^\infty
\varpi_{q,p}(n)=\te(p,q)<\infty,&
\\
%
\label{beta1}
&\displaystyle \sum_{r=0}^\infty\bigl[
\beta(q, r)\bigr]^\delta<\infty,&
\\
%
\label{eqsix5}
&\displaystyle \gamma_{m}<\infty,\qquad \gamma_{2q\iota}<\infty
\qquad\mbox{with } \frac{1}{2}\ge\frac{1}{p}+\frac{\iota+2}{m}+
\frac{\delta}{q}.&
\end{eqnarray}
\end{assumption}
%
\begin{remark}\label{simple}
The reader willing to reduce technicalities in the first reading can be advised
to keep in mind simplified assumptions such as $\ell=k$ [i.e., to
consider only
linear times $q_j(n)=jn$], bounded and Lipschitz continuous $F$ and
$\varpi_{q,p}(n), \be(q,n)$ decaying exponentially fast in $n$. Such
simplifications save some of our estimates but, otherwise, most of our
machinery still should be applied.
\end{remark}

In order to give a detailed statement of our main result as well as for its
proof, it will be essential to represent the function $F=
F(x_1,x_2,\ldots,
x_\ell)$ in the form
%
\begin{equation}
\label{F} F=F_1(x_1)+\cdots+F_\ell(x_1,
x_2,\ldots, x_\ell),
\end{equation}
where for $i<\ell$,
%
\begin{eqnarray}
\label{18}
F_i(x_1,\ldots, x_i)&=&\int
F(x_1,x_2,\ldots, x_\ell) \,d\mu
(x_{i+1})\cdots d\mu(x_\ell)
\nonumber\\[-8pt]\\[-8pt]
&&{} -\int F(x_1,x_2,\ldots, x_\ell)  \,d
\mu(x_i)\cdots d\mu (x_\ell)
\nonumber
\end{eqnarray}
and
\[
F_\ell(x_1,x_2,\ldots, x_\ell)=F(x_1,x_2,\ldots, x_\ell) -\int F(x_1,x_2,\ldots,
x_\ell)  \,d\mu(x_\ell),
\]
which ensures, in particular, that
%
\begin{equation}
\label{F0} \int F_i(x_1, x_2,\ldots,x_{i-1}, x_i) \,d\mu(x_i)\equiv0
\qquad\forall x_1, x_2,\ldots, x_{i-1}.
\end{equation}
These enable us to write
%
\begin{equation}
\label{xi} \xi_N(t)=\sum_{i=1}^k
\xi_{i,N}(it)+\sum_{i=k+1}^\ell
\xi_{i,N}(t),
\end{equation}
where for $1\leq i\leq k$,
%
\begin{equation}
\label{def21} \xi_{i,N}(t)=\frac{1}{\sqrt N}\sum
_{n=1}^{[{Nt}/{i}]} F_i\bigl(X(n), X(2n),\ldots,
X(in)\bigr)
\end{equation}
and for $i\ge k+1$,
%
\begin{equation}
\label{def22} \xi_{i,N}(t)=\frac{1}{\sqrt N}\sum
_{n=1}^{[Nt]} F_i\bigl(X
\bigl(q_1(n)\bigr),\ldots, X\bigl(q_i(n)\bigr)\bigr).
\end{equation}

\begin{theorem}\label{MainThm}
Suppose that Assumption \ref{Hyp2} holds true. Then the $\ell$-dimen\-sional
process $\{\xi_{i,N}(t)\dvtx 1\le i\le\ell\}$
converges in distribution as $N\to\infty$ to a Gaussian process
$\{\eta_i(t)\dvtx  1\le i\le\ell\}$ with stationary independent increments.
The means are $0$ and the covariances are given by $ E[ \eta_i(s)\eta_j(t)]
=\min(s,t) D_{i,j}$. For $i,j\leq k$,
$D_{i,j}$ is given by Proposition \ref{covariance}. Moreover,
$D_{i,j}=0$ if
$i\not=j$, and either $i$ or $j$ is at least $k+1$,
making the processes $\{\eta_i(\cdot),  i\geq k+1\}$ independent of each
other and of
$\{\eta_j(\cdot)\dvtx  j\le k\}$. For $i\ge k+1$, the variance of $\eta_i(t)$ is
given by $t D_{i,i}$, where
\[
D_{i,i}=\int\bigl|F_i(x_1,x_2,\ldots,x_i)\bigr|^2\,d
\mu(x_1)\,d\mu(x_2)\cdots d\mu(x_i).
\]
Finally, the distribution of the process $\xi_N(\cdot)$ converges to the
Gaussian
process $\xi(\cdot)$ which can be represented
in the form
%
\begin{equation}
\label{eqmain1} \xi(t)=\sum_{i=1}^k
\eta_i(it)+\sum_{i=k+1}^\ell
\eta_i(t).
\end{equation}
If $k\geq2$, then the process $\xi(t)$ may not have independent increments.
\end{theorem}

In order to understand our assumptions, observe that $\varpi_{q,p}$
is clearly nonincreasing in $q$ and nondecreasing in $p$. Hence,
for any pair $p,q\geq1$,
\[
\varpi_{q,p}(n)\leq\psi(n).\vadjust{\goodbreak}
\]
Furthermore, by the real version of the Riesz--Thorin interpolation
theorem or the Riesz convexity theorem (see \cite{Ga}, Section 9.3,
and \cite{DS}, Section VI.10.11), whenever $\theta\in[0,1],  1\leq
p_0,p_1,q_0,q_1\leq\infty$ and
\[
\frac1p=\frac{1-\theta}{p_0}+\frac\theta{p_1},\qquad  \frac1q=
\frac
{1-\theta}{q_0}+\frac\theta{q_1},
\]
then
%
\begin{equation}
\label{eq24} \varpi_{q,p}(n)\le2\bigl(\varpi_{q_0,p_0}(n)
\bigr)^{1-\theta} \bigl(\varpi_{q_1,p_1}(n)\bigr)^\theta.
\end{equation}
In particular, using the obvious bound $\varpi_{q_1,p_1}\leq2$
valid for any $q_1\geq p_1$, we obtain from (\ref{eq24}) for pairs
$(\infty,1)$, $(2,2)$ and $(\infty,\infty)$ that for all $q\geq
p\geq1$,
%
\begin{eqnarray}
\label{eqsix2} \varpi_{q,p}(n)&\le&\bigl(2\alpha(n)
\bigr)^{{1/p}-{1/q}},\nonumber\\
\varpi_{q,p}(n)&\le&2^{1+1/p-1/q}\bigl(\rho(n)
\bigr)^{1-1/p+1/q}
\quad\mbox{and}\\
\varpi_{q,p}(n)&\le&2^{1+1/p}\bigl(\phi(n)
\bigr)^{1-1/p}.
\nonumber
\end{eqnarray}
We observe also that by the H\"older inequality for $q\geq p\geq1$
and $\alpha\in(0,p/q)$,
%
\begin{equation}
\label{eq201} \beta(q,r)\le2^{1-\alpha} \bigl[\beta(p,r)
\bigr]^\alpha\gamma^{1-\al
}_{{pq(1-
\al)}/({p-q\al})}
\end{equation}
with $\gamma_\theta$ defined in (\ref{2m1}). Thus, we can formulate
Assumption \ref{Hyp2} in terms of more familiar $\alpha, \rho,
\phi$,
and $\psi$-mixing coefficients and with various moment conditions. It
follows also from (\ref{eq24}) that if $\varpi_{q,p}(n)\to0$ as
$n\to\infty$
for some $q>p\geq1$, then
%
\begin{equation}
\label{eq25} \varpi_{q,p}(n)\to0  \qquad\mbox{as } n\to\infty\mbox{ for
all } q> p\geq1,
\end{equation}
and so (\ref{eq25}) holds true under Assumption \ref{Hyp2}.

Concerning the function $F=F(x_1,\ldots,x_\ell)$, we can take it, for instance,
to be a polynomial in $x_1,\ldots,x_\ell$, in particular,
$F(x_1,\ldots,x_\ell)=
x_1x_2\cdots x_\ell$ which leads to a functional central limit theorem for
\[
N^{-1/2}\sum_{1\leq n\leq[Nt]}X\bigl(q_1(n)
\bigr)X\bigl(q_2(n)\bigr)\cdots X\bigl(q_\ell(n)\bigr).
\]

The key point of our proof will be construction of martingale approximations
for the processes $\xi_{i,N}(t)$'s, where we will have to overcome problems
imposed by strong dependencies between terms in the sum (\ref{24d}),
as well as between arguments $X(q_j(n)),  j=1,2,\ldots,\ell$, of the function
$F$ there. The
realignment in the definition of $\{\xi_{i,N} (t)\}$ for $i\le k$ will
also be
important since it makes the collection a process with independent increments
in the limit. Otherwise, in the limit, increments of $\{\xi_i(t)\}$
will be
correlated with the increments of $\{\xi_j(t)\}$ at different time
points. It
will not matter for $i\ge k+1$, for they will all turn out to be mutually
independent in the limit.

The conditions of Theorem \ref{MainThm} hold true for
many important models. Let, for instance, $\Up_n$ be a Markov chain on a
space $M$ satisfying the Doeblin condition (see, e.g.,
\cite{IL}, pages 367 and 368) and $f_j,  j=1,\ldots,\ell$, be bounded measurable
functions on the space of sequences $x=(x_i,  i=0,1,2,\ldots,  x_i\in
M)$ such
that $|f_j(x)-f_j(y)|\leq Ce^{-cn}$ provided $x=(x_i),  y=(y_i)$
and $x_i=y_i$ for all $i=0,1,\ldots,n$, where $c,C>0$ do not depend on
$n$ and $j$. In fact, some polynomial decay in $n$ will suffice here as well.
Let $X(n)=(X_1(n),\ldots,X_\ell(n))$ with $X_j(n)=
f_j(\Up_n,\Up_{n+1},\Up_{n+2},\ldots)$ and take $\sig$-algebras $\cF_{kl},
k<l$ generated by $\Up_k,\Up_{k+1},\ldots,\Up_l$, then our condition
will be satisfied considering $\{\Up_n,  n\geq0\}$ with its invariant
measure as a stationary process. In fact, our conditions hold true for a
more general class of processes, in particular, for Markov chains whose
transition operator has an $L^2$ spectral gap which leads to an exponentially
fast decay of the $\rho$-mixing coefficient.
%
\begin{remark}\label{rem25}
Formally, (\ref{21}) requires some stationarity and, for instance, if we
consider a Markov chain $\xi_n$ satisfying the Doeblin condition but
whose initial distribution differs from its invariant measure, then
(\ref{21}) does not hold true for $X(n)=f(\xi_n)$. Still, a slight
modification makes our method to work so that Theorem \ref{MainThm}
(as well as its continuous time version Theorem
\ref{ContThm}) remain valid. In order to do this, we consider another
probability measure $\Pi$ on the space $(\Om,\cF)$ and require the
weak stationarity (\ref{21}) with respect to $\Pi$, that is,
$X(n)\Pi=
\mu$ and $(X(n),X(n'))\Pi=\mu_{n-n'}$. In addition, we modify the
definition of the dependence coefficient $\varpi_{q,p}$ in (\ref{eq21}),
taking the\vspace*{1pt} conditional expectation of $g$ there with respect to the
probability $P$ while taking the unconditional expectation of $g$ with
respect to $\Pi$. It is easy to see that under the same assumptions
as above but with modified (\ref{eq21}) and (\ref{21}) our proof
will still go through.
\end{remark}

Important classes of processes satisfying our conditions come from
dynamical systems. Let $T$ be a $C^2$ Axiom A diffeomorphism (in
particular, Anosov) in a neighborhood of an attractor or let $T$ be
an expanding $C^2$ endomorphism of a Riemannian manifold $M$ (see
\cite{Bo}), $f_j$'s be either H\"older continuous functions or functions
which are constant on elements of a Markov partition and let
$X(n)=(X_1(n),\ldots,X_\ell(n))$ with $X_j(n)=f_j(T^nx)$. Here the probability
space is $(M,\cB,\mu)$, where $\mu$ is a Gibbs invariant measure
corresponding
to some H\"older continuous function and $\cB$ is the Borel $\sig
$-field. Let
$\zeta$ be a finite Markov partition
for $T$, then we can take $\cF_{kl}$ to be the finite $\sig$-algebra
generated by the partition $\bigcap_{i=k}^lT^i\zeta$. In fact, we can
take here not only H\"older continuous $f_j$'s but also indicators
of sets from $\cF_{kl}$. A related example corresponds to $T$ being
a topologically mixing subshift of finite type, which means that $T$
is the left shift on a subspace $\Xi$ of the space of one-sided
sequences $\vs=(\vs_i,i\geq0), \vs_i=1,\ldots,l_0$, such that $\vs\in
\Xi$
if $\pi_{\vs_i\vs_{i+1}}=1$ for all $i\geq0$ where\vadjust{\goodbreak} $\Pi=(\pi_{ij})$
is an $l_0\times l_0$
matrix with $0$ and~$1$ entries and such that $\Pi^n$ for some $n$
is a matrix with positive entries. Again, we have to take in this
case $f_j$ to be bounded H\"older continuous [with respect to the
metric $d((\vs_i,i\geq0), (\vs^\prime_i,i\geq0))=\exp(-\min\{
j\geq0\dvtx
\vs_j\ne\vs^\prime_j\})$] functions on the
sequence space above, $\mu$ to be a Gibbs invariant measure
corresponding to some H\"older continuous function and to define
$\cF_{kl}$ as the finite $\sig$-algebra generated by cylinder sets
with fixed coordinates having numbers from $k$ to~$l$. The
exponentially fast $\psi$-mixing is well known in
the above cases (see~\cite{Bo}). Among other dynamical systems with
exponentially fast $\psi$-mixing we can mention also the Gauss map
$Tx=\{1/x\}$ (where $\{\cdot\}$ denotes the fractional part) of the
unit interval with respect to the Gauss measure $G$
(see \cite{IL} and~\cite{He}). The latter enables us to consider the number
$N_a(x,n)$, $a=(a_1,\ldots,a_\ell)$
of $m$'s between 0 and $n$ such that the $q_j(m)$th digit of the
continued fraction of $x$ equals certain integer $a_j,j=1,\ldots,\ell$. Then
Theorem \ref{MainThm} implies a central limit theorem for $N_a(x,n)$
considered
as a random variable on the probability space $((0,1],\cB, G)$. In
fact, our
results rely only on sufficiently fast
$\al$ or $\rho$-mixing which holds true for wider classes of
dynamical systems,
in particular, those whose transfer operator has an $L^2$ spectral gap
(such as many one-dimensional
not necessarily uniformly expanding maps) which ensures an
exponentially fast
$\rho$-mixing. Of course, there are many stationary processes (including
unbounded ones) and dynamical
systems with polynomially fast mixing which still satisfy our
conditions, but
they are more difficult to describe in short.

Next, we discuss a continuous time version of our theorem.
Our continuous time setup consists of a $\wp$-dimensional process
$X(t),  t\geq0$ on a probability space $(\Om,\cF,P)$ and of a
family of
$\sig$-algebras $\cF_{st}\subset\cF, -\infty\leq s\leq t\leq
\infty$ such
that $\cF_{st}\subset\cF_{s't'}$ if $s'\leq s$ and $t'\geq t$. We
assume that
the distribution of $X(t)$ is independent of $t$ and denote it by $\mu
$. The
joint distribution of $\{X(t), X(t+s)\}$ is assumed to depend only on
$s$ and
is denoted by $\mu_s$. For all $t\geq0$ we set
%
\begin{equation}
\label{varpi} \varpi_{q,p}(t)=\sup_{s\geq0}\varpi_{q,p}(
\cF_{-\infty,s},\cF_{s+t,\infty})
\end{equation}
and
%
\begin{equation}
\label{beta} \beta(p,t)=\sup_{s\geq0}\bigl\|X(s)-E \bigl[X(s)|
\cF_{s-t,s+t} \bigr]\bigr\|_p,
\end{equation}
where $\varpi_{q,p}(\cG,\cH)$ is defined by (\ref{eq21}). We
continue to
impose Assumption \ref{Hyp2} on the decay rates of
$\varpi_{q,p}(t)$ and $\beta(p,t)$. Although they only involve integer
values of $t$, it will suffice since they are nonincreasing functions of
$t$. Let
$q_1(t)< q_2(t) <\cdots< q_\ell(t)$ be increasing positive functions
such that $q_i(t)=i t$ for $i=1,\ldots,k$ while $q_i(t),  i>k$ grow faster
in $t$. We assume that these functions satisfy the conditions (\ref{2q2})
and (\ref{2q3}) (with $t$ in place of~$n$), while (\ref{2q1}) is replaced
by
%
\begin{equation}
\label{new2q} \lim_{t\to\infty}\bigl(q_i(t+{\gamma})-q_i(t)
\bigr)=\infty  \qquad\mbox{for any }  {\gamma}>0
\mbox{ and } i>k.
\end{equation}

\begin{theorem}\label{ContThm}
Suppose that Assumption \ref{Hyp2} holds true. Then the distribution
of the
process
%
\begin{equation}
\label{cont} \xi_N(t)=\frac1{\sqrt N}\int_0^{Nt}F
\bigl(X\bigl(q_1(s)\bigr),\ldots,X\bigl(q_\ell (s)\bigr)
\bigr)\,ds
\end{equation}
on $C[0,T]$ converges to the distribution of a Gaussian process
$\xi(t)$ which has the representation (\ref{eqmain1}), but, unlike in
the discrete time case, all processes $\eta_i,  i>k$ are zero there
while $\{\eta_1(t),\ldots,\eta_k(t)\}$ is a $k$-dimensional Gaussian
process having stationary independent increments. The means are $0$ and
variances and covariances are given by
$E[\eta_i(s)\eta_j(t)]=\min(s,t)D_{i,j},  i,j=1,\ldots,k$. The expressions
for these $D_{i,j}$ are provided in Section \ref{sec7}.
\end{theorem}

The conditions of Theorem \ref{ContThm} are satisfied when, for instance,
$X(t)=(X_1(t),\ldots,X_\wp(t))$ with $X_j(t)=f_j(\Up_t)$, where $\Up_t$
is either
an irreducible continuous time finite state Markov chain or a nondegenerate
diffusion process on a compact manifold. Furthermore, Ornstein--Uhlenbeck
type processes $X(t)$ produce a class of unbounded processes still satisfying
our assumptions.
On the other hand, these conditions do not usually hold true
for important classes of continuous time dynamical systems (flows) having
rich probabilistic properties such as Axiom A (in particular, Anosov) flows
where in the proof of conventional central limit theorems the standard tool of
suspension flows is usually applied while this does
not seem to work in our circumstances and a different approach should be
employed here.
%
\begin{remark}\label{momentrem}
Under stronger mixing and moment conditions it is possible to derive
convergence of all moments of $\xi_N(t)$ to the corresponding moments
of the limiting Gaussian process $\xi(t)$.
\end{remark}

\section{Approximation estimates}\label{sec3}

This section contains estimates which are crucial for our proofs and
some of them may also have independent interest beyond this paper. Still,
in the first reading the reader can skip this section all together and
only refer to general estimates of Corollary \ref{cor} when needed in what
follows.
We will make repeated use of the following simple variations of H\"{o}lder's
inequality.
%
\begin{lemma}\label{L50}
\textup{(i)} For any two random variables $Z,D$,
\[
\bigl\|Z^h D^\ka\bigr\|_a\le \|Z\|^h_{a^\ast}
\|D\|_{b^\ast}^\ka
\]
provided $\frac{1}{a}\ge\frac{h}{a^\ast}+\frac{\ka}{b^\ast}$.
If, in addition, $|D|\le|Z|$ a.e. (almost everywhere), we can replace
$\ka$
by $\al\le\ka$ and change $h$ to $h+\ka-\al$, obtaining
\[
\bigl\| Z^h D^\ka\bigr\|_a\le\bigl\| Z^{h+\ka-\al}
D^\al\bigr\|_a\le\|Z\|^{h+\ka-\al
}_{a^\ast} \|D
\|_{b^\ast}^{\al}
\]
provided $\frac{1}{a}\ge\frac{h+\ka-\al}{a^\ast}+\frac{\al
}{b^\ast}$.\vadjust{\goodbreak}

\mbox{}\hphantom{\textup{i}}\textup{(ii)} If $f(x,\omega)$ is a measurable function of $x$ and $\omega$
such that
for almost all ${\omega}$,
\[
\bigl|f(x,\omega)\bigr|\le C(\omega)\bigl[1+|x|^h\bigr],
\]
then
\[
\bigl\|f\bigl(X(\omega),\omega\bigr)\bigr\|_a\le \bigl(1+\gamma^h_m
\bigr) \bigl\|C(\omega)\bigr\|_p
\]
provided $\frac{1}{a}\ge\frac{1}{p}+\frac{h}{m}$ where $\gamma_m$
is a bound
for $\|X\|_m$.

\textup{(iii)} If $f(x,\omega)$ is a measurable function of $x$ and $\omega$ satisfying
for almost all $\omega$,
\[
\bigl|f(x,\omega)-f(y,\omega)\bigr|\le H(\omega)\bigl[1+|x|^h+|y|^h
\bigr]|x-y|^\delta,
\]
then
%
\begin{equation}
\label{eqfive1} \bigl\|f\bigl(X(\omega),\omega\bigr)-f\bigl(Y(\omega),\omega\bigr)
\bigr\|_a\le\bigl(1+2\gamma_{m}^h\bigr) \bigl\|H(\omega)
\bigr\|_p \|X-Y\|_{q}^\delta
\end{equation}
provided $\frac{1}{a}\ge\frac{1}{p}+\frac{h}{m}+\frac{\delta}{q}$ where
$\gamma_m$ is a bound for $\|X\|_m$ and $\|Y\|_m$.
\end{lemma}
\begin{pf} For (i), by H\"{o}lder's inequality,
\[
\bigl\|Z^hD^\ka\bigr\|_a= \bigl[E\bigl[Z^{ah}
D^{a\ka}\bigr] \bigr]^{1/a} \le\|Z\|^h_{a^\ast}
\|D\|^\ka_{b^\ast}
\]
provided $\frac{1}{a}\ge\frac{h}{a^\ast}+\frac{\kappa}{b^\ast}$.
If $|D|\le
|Z|$ and $0 \le\al\le\ka$,
\[
\bigl\|D^\ka Z^h\bigr\|_a\le \bigl\|D^\al
Z^{h+\ka-\al}\bigr\|_a\le \|Z\|^{(h+\ka-\al
)}_{a^\ast} \|D
\|_{b^\ast}^{\al}
\]
provided $\frac{1}{a}\ge\frac{h+\ka-\al}{a^\ast}+\frac{\al
}{b{^\ast}}$.\vspace*{1pt}

For (ii), by H\"{o}lder's inequality,
\begin{eqnarray*}
E\bigl[\bigl|f\bigl(X(\omega), \omega\bigr)\bigr|^a\bigr]
&\le& E \bigl[
\bigl[C(\omega)\bigr]^a  \bigl[1+|X|^h
\bigr]^a \bigr]
\\
&\le& \bigl[ E \bigl[ \bigl[C(\omega)\bigr]^p \bigr]
\bigr]^{a/p} \bigl[ E \bigl[\bigl[1+ |X|^h
\bigr]^{p^\ast/h} \bigr] \bigr]^{ah/p^\ast}
\end{eqnarray*}
provided $\frac{1}{a}\ge\frac{1}{p}+\frac{h}{p^\ast}$.\vspace*{1pt}

The assertion (iii) follows similarly from the inequality
\[
E\bigl[|XYZ|\bigr]\le\|X\|_{s_1}\|Y\|_{s_2}\|Z\|_{s_3},
\]
if $1\ge\frac{1}{s_1}+\frac{1}{s_2}+\frac{1}{s_3}$.
\end{pf}

We will need also the following.
%
\begin{lemma}\label{L21} \textup{(i)} Let $F(x_1,\ldots, x_{\ell-1}, x_\ell
)$ be any
function that satisfies (\ref{24}) and (\ref{24b}). Then the functions
$F_i(x_1,\ldots, x_i)$ defined in (\ref{18}) will inherit similar properties
from $F$.

\textup{(ii)} Let $Z$ be a random vector in $L_\iota(P)$ with $\|Z\|_{\iota
}\le
\gamma_\iota$ and $\cG\subset\cF$ be a sub $\sigma$-field. If
\[
G_i(x_1,\ldots, x_{i-1}, \omega)=E
\bigl[F_i\bigl(x_1,\ldots, x_{i-1}, Z(\omega)
\bigr)|\cG\bigr],
\]
then
\[
\bigl|G_i(x_1,\ldots, x_{i-1}, \omega)\bigr|\le C
\bigl(1+C(\omega)^\iota +|x|^{\iota}\bigr)
\]
and
\begin{eqnarray*}
&&
\bigl|G_i(x_1,\ldots, x_{i-1},
\omega)-G_i(y_1,\ldots, y_{i-1}, \omega )\bigr|
\\
&&\qquad\le C \bigl( 1+ C(\omega)^\iota+|x|^\iota+|y|^{\iota}
\bigr) |x-y|^\ka,
\end{eqnarray*}
where $C>0$ is a constant, $C(\omega)= (2E[|Z|^\iota|\cG] )^{1/\iota}$ and $\|C(\omega)\|^\iota_\iota\le 2\gamma_\iota^\iota$.
\end{lemma}
\begin{pf}
For (i), if
\[
\bigl|F(x_1,x_2,\ldots, x_i)\bigr|\le
C_1\bigl(C_2+|x|^\iota\bigr),
\]
then
\begin{eqnarray*}
\biggl|\int F(x_1,\ldots, x_{i-1}, x_i) \,d
\mu(x_i)\biggr|&\le& \int\bigl|F(x_1,\ldots, x_{i-1},
x_i)\bigr| \,d\mu(x_i)
\\
&\le& C_1 \bigl(C_2+|x|^\iota+
\gamma^\iota_\iota \bigr).
\end{eqnarray*}
The H\"{o}lder property is similar.

Assertion (ii) follows from
\[
\bigl|G_i(x_1,\ldots, x_{i-1}, \omega)\bigr|\le E
\bigl[\bigl|F_i(x_1,\ldots, x_{i-1}, Z)\bigr| \vert\cG
\bigr] \le C_1E\bigl[ \bigl(C_2+|x|^{\iota}+|Z|^{\iota}
\bigr)|\cG\bigr]
\]
and
\begin{eqnarray*}
&&\bigl|G_i(x_1,\ldots, x_{i-1},
\omega)-G_i(y_1,\ldots, y_{i-1}, \omega)\bigr| \\
&&\qquad\le E
\bigl[\bigl|F_i(x_1,\ldots, x_{i-1}, Z)
-F_i(y_1,\ldots, y_{i-1}, Z)\bigr| \vert\cG
\bigr] \\
&&\qquad\le CE \bigl[ \bigl(1+|x|^\iota+|y|^\iota+
2|Z|^\iota\bigr) \vert\cG \bigr]|x-y|^\ka.
\end{eqnarray*}
\upqed\end{pf}
%
\begin{remark}
Here and in what follows it is sometimes more convenient to use together
with (\ref{24}) and (\ref{24b}) also slightly different-looking conditions
for growth and H\"older continuity of functions we are dealing with
(i.e., considering $|x|^\iota$ in place of $\sum_{j=1}^\ell
|x_j|^\iota,
x\in\bbR^{\ell\wp}$), but,
in fact, these sets of conditions are equivalent since for any $b_1,b_2,\ldots,
b_l\geq0$ and ${\gamma}>0$,
%
\begin{equation}
\label{bgamma} \max_{1\leq i\leq l}b_i^{\gamma}\leq\sum
^l_{i=1}b_i^{\gamma}\leq
l\max_{1\leq i\leq l} b_i^{\gamma}\leq l\Biggl(\sum
^l_{i=1}b_i\Biggr)^{\gamma}\leq
l^{1+{\gamma}}\max_{1\leq i\leq l}b_i^{\gamma}.
\end{equation}
\end{remark}

We will need the following result which will serve as a base for our estimates
and is, in fact, an extended multidimensional version of the standard
Kolmogorov theorem on the H\"older continuity of sample paths.
%
\begin{theorem}\label{kol}
Let $f(x, \omega)$ be a collection of random variables continuously
(or separable) dependent on $x\in\bbR^d$ for almost all $\omega$ and
satisfying
%
\begin{eqnarray}
\label{f1} \bigl\|f( x,\omega)-f( y,\omega)\bigr\|_{p}&\le& C_1
\bigl(1+|x|^\iota+ |y|^\iota \bigr)|x-y|^\ka
\quad\mbox{and}\nonumber\\[-8pt]\\[-8pt]
\bigl\|f(x,\omega)\bigr\|_{p}&\le& C_2
\bigl(1+|x|^\iota\bigr)\nonumber
\end{eqnarray}
with $\ka>\frac{d}{p}$. Then for any $\iota'> \iota+\frac{d}{ p}$
and $\te$ such that $\ka> \theta>\frac{d}{p}$ there is a random variable
$G(\omega)$ such that
%
\begin{eqnarray}
\label{G} \bigl|f(x,\omega)\bigr|\le G(\omega)
\bigl(1+|x|^{\iota'}\bigr)\nonumber\\[-8pt]\\[-8pt]
&&\eqntext{\mbox{a.e. }   \mbox{ with }   \bigl\|G(\omega)\bigr\|_{p} \le
c_0[C_1+C_2]^{{d}/({p\te})}
 C_2^{1-{d}/({p\te})},}
\end{eqnarray}
where $c_0=c_0(d, p, \ka, \te, \iota, \iota')>0$ depends only on
parameters in
brackets. Since $\ka\leq1$ and $p\ka>d$, it follows that
$p>d$ and, therefore, we can always take $\iota'=\iota+1$.
Furthermore, if $Z\in
L_m(P)$ is a random variable with values in $\bbR^d$ satisfying $\|Z\|_m\le
\gamma_m$ and if $\frac{1}{a}\ge\frac{1}{p}+\frac{\iota+1}{m}$, then
%
\begin{eqnarray}
\label{eq54} \bigl\|f\bigl(Z(\omega),\omega\bigr)\bigr\|_{a}&\le&\bigl\| G(\omega)
\bigl(1+|Z|^{\iota+1}\bigr)\bigr\|_a
\nonumber\\
&\le& c_0[C_1+C_2]^{{d}/({p\te})}C_2^{1-{d}/({p\te})}
\bigl[1+\gamma^{\iota+1}_m\bigr] \\
&=&c_0c(
\gamma_m)[C_1+C_2]^{{d}/({p\te})}C_2^{1-{d}/({p\te})}.
\nonumber
\end{eqnarray}
If $p(\ka-\delta)>d$, then we can have an almost sure H\"{o}lder estimate
\[
\bigl|f(x,\omega)-f(y,\omega)\bigr|\le H(\omega)\bigl[1+|x|^{\iota+2}+|y|^{\iota+2}
\bigr] |x-y|^\delta
\]
with
\[
\bigl\|H(\omega)\bigr\|_p\le c(\ka,\te, d, p, \delta, \iota)
(C_1+C_2)
\]
and the estimate
%
\begin{eqnarray}
\label{XY}
&&\bigl\|f(X_1, X_2,\ldots, X_{i-1},
\omega)-f(Y_1, Y_2,\ldots, Y_{i-1},\omega)
\bigr\|_a
\nonumber\\
&&\qquad\le\bigl\|H (\omega)\bigl[1+|X|^{\iota+2}+|Y|^{\iota+2}\bigr]
|X-Y|^\delta\bigr\|_a
\\
&&\qquad\le\|H\|_{p}\bigl(1+\gamma_{m}^{\iota+2}\bigr)\sum
_{j=1}^{i-1}\|X_j-Y_j
\|_{q}^\delta
\nonumber
\end{eqnarray}
provided $\frac{1}{a}\ge\frac{1}{p}+\frac{\iota+2}{m}+\frac
{\delta}{q}$,
where $X=(X_1,\ldots,X_{i-1}),  Y=(Y_1,\ldots,Y_{i-1})\in\bbR^d$ and
$X_j,Y_j,
j=1,\ldots,i-1$, are random vectors with $\| X\|_m,\| Y\|_m\leq{\gamma}_m$.
\end{theorem}
%
\begin{remark}
There are several types of constants that we need to keep track of. Constants
$C,K$ will be absolute and may change from line to line. Constants $c$
will depend on other parameters like moments and will be denoted by
$c(\cdot)$ to indicate this dependence.
\end{remark}
\begin{pf*}{Proof of Theorem \ref{kol}}
For $\iota'=\iota+1>\iota+\frac dp$ set
\[
\tilde f(x,{\omega})=f(x,{\omega}) \bigl(1+|x|^{\iota+1}
\bigr)^{-1}.
\]
Then by (\ref{f1}), if $|x-y|\le\rho_0=\frac{\sqrt d}{2}$,
%
\begin{eqnarray}
\label{f2}
&&\bigl\|\tilde f(x,{\omega})-\tilde f(y,{\omega})\bigr\|_p\nonumber\\
&&\qquad\leq\bigl\|
f(x,{\omega })-f(y,{\omega})\bigr\|_p \bigl(1+|x|^{\iota+1}
\bigr)^{-1}
\nonumber\\[-8pt]\\[-8pt]
&&\qquad\quad{}+\bigl\| f(y,{\omega})\bigr\|_p \bigl||y|^{\iota+1}-|x|^{\iota+1}
\bigr|\eta(x) \nonumber\\
&&\qquad\leq c_1[C_1+C_2]|x-y|^\ka
\eta(x)
\nonumber
\end{eqnarray}
and
%
\begin{equation}
\label{f3} \bigl\|\tilde f(x,{\omega})\bigr\|_p\leq C_2\eta(x),
\end{equation}
where $\eta(x)=(1+|x|^\iota)(1+|x|^{\iota+1})^{-1}$ and
$c_1=c_1(\iota,
\ka,d)<\infty$ is a constant depending only on the parameters in
brackets. Let
$B_w(\rho)$ denote an open unit ball of radius $\rho$ centered at
$w\in\bbR^d$. A multivariate generalization of a result of Garsia, Rodemich
and Rumsey (see \cite{SV}, page 60) states that if a continuous (or separable)
$g\dvtx  \bbR^d \to\bbR$ satisfies
\[
\int_{B_w(\rho)\times B_w(\rho)}\Psi \biggl(\frac{|g(x)-g(y)|}{\sig
(|x-y|)} \biggr) \,dx \,dy \le
Q_{w,\rho}
\]
for some continuous strictly increasing functions $\Psi,\sig$ with
$\sig(0)
=\Psi(0)=0$, then for any $x,y\in B_w(\rho)$,
%
\begin{equation}
\label{GRR} \bigl|g(x)-g(y)\bigr|\le8\int_0^{2|x-y|}
\Psi^{-1} \biggl(\frac
{4^{d+1}Q_{w,\rho}} {
k_d u^{2d}} \biggr)\,d\sig(u),
\end{equation}
where $k_d=\inf_{a\in B_w(\rho),0<u\le2}\frac{|B_a(u)\cap B_0(1)|} {
u^d}$. Choose here $\Psi( z)=|z|^p$ and $\sig(u)=u^{\theta+{2d}/{p}}$
with $0<\te<\ka-\frac dp$ and set
\[
\bigl[Q_{w,\rho}({\omega})\bigr]^p=\int_{B_w(\rho)\times B_w(\rho)}
\frac
{|\tilde f(x,{\omega})-
\tilde f(y,{\omega})|^p}{|x-y|^{p\te+2d}} \,dx \,dy.
\]
Then by the result above together with (\ref{f2}) we derive that there
exists $c_2=c_2(\iota,\iota',\ka,\te,p,d)>0$ such that for any
$x,y\in
B_w(\rho)$,
%
\begin{equation}
\label{f4} \bigl|\tilde f(x,\omega)-\tilde f(y,\omega)\bigr|\le c_2
 Q_{w,\rho}({\omega })|x-y|^\te
\end{equation}
and for $0 < \rho\le\rho_0$,
%
\begin{equation}
\label{f5} \|Q_{w,\rho}\|_p \leq c_2
v_d (C_1+C_2) \eta(w) \rho^{(\ka-\te)},
\end{equation}
where
\[
v^p_d=\int_{B_0(1)\times B_0(1)}
|x-y|^{\ka p-p\te-2d}  \,dx  \,dy<\infty
\]
provided $p(\ka-\te)>d$. Observe that (\ref{f4})
and (\ref{f5}) are, in fact, the conclusion of a multidimensional
version of
the Kolmogorov theorem (see, e.g., \cite{Ku}, Theorem 1.4.1), but
our argument relies also on the specific estimate (\ref{f5}).

Let $ \bbZ_h^d$ be the lattice in $\bbR^d$ with spacing $h$. The maximum
distance of any point in $\bbR^d$ from $\bbZ_h^d$ is
$h\frac{\sqrt d}{2}=h\rho_0$. Therefore, in the cube of side $h$ centered
around $w\in\bbZ_h^d$ we have
\[
\bigl|\tilde f(x,\omega)\bigr|\le\bigl|\tilde f(w,{\omega})\bigr|+ c_2
Q_{w, h\rho
_0}({\omega}) \rho_0^\te
h^\te
\]
and so
\[
\bigl|\tilde f(x,\omega)\bigr|^p\le2^{p-1}\bigl[\bigl|\tilde f(w,{
\omega})\bigr|^p+ c^p_2 Q^p_{w, h
\rho_0}({
\omega})\rho_0^{p\te} h^{p\te}\bigr].
\]
Therefore,
\begin{eqnarray*}
\sup_{x\in\bbR^d} \bigl|\tilde f(x,\omega)\bigr|^p&\le&2^{p-1}
\sup_{w\in
\bbZ_h^d} \bigl[\bigl|\tilde f(w,{\omega})\bigr|^p+c^p_2
Q^p_{w, h\rho_0}({\omega})\rho_0^{p\te}
h^{p\te} \bigr]
\\
&\le& 2^{p-1} \sum_{w\in\bbZ_h^d} \bigl[\bigl|\tilde
f(w,{\omega })\bigr|^p+c^p_2
Q^p_{w,
h\rho_0}({\omega})\rho_0^{p\te}
h^{p\te} \bigr]
\end{eqnarray*}
and, using (\ref{f5}) together with the estimate
$\sum_{w\in\bbZ_h^d} [\eta(w)]^p\le c^p_4(d,i,p) h^{-d}$,
\begin{eqnarray*}
E \Bigl[\Bigl[\sup_{\bbR^d} \bigl|\tilde f(x,\omega)\bigr|^p\Bigr]
\Bigr]&\le&2^{p-1} \sum_{w\in
\bbZ_h^d}\bigl\|\tilde f(w,{
\omega})\bigr\|^p_p
\\
&&{}+2^{p-1}c^p_2 \rho_0^{p\te}
h^{p\te} \sum_{w\in\bbZ_h^d}\bigl\|Q_{w, h\rho_0}({
\omega})\bigr\|^p_p
\\
&\le& c_3^p\bigl[C^p_2+(C_1+C_2)^p
h^{p\ka}\bigr] \sum_{w\in\bbZ_h^d}\bigl[\eta(w)
\bigr]^p \\
&\le& c_5^p \bigl[C^p_2+(C_1+C_2)^p
h^{p\ka}\bigr]h^{-d}
\end{eqnarray*}
with a constant $c_5=c_5(d,p,\iota,\ka,\te)>0$.
Making the choice of
$h=\break [\frac{C_2}{C_1+C_2}]^{1/\ka}\le1$,
\[
E \Bigl[\Bigl[\sup_{\bbR^d} \bigl|\tilde f(x,\omega)\bigr|^p\Bigr]
\Bigr]\le c_6^p C^{p-
{d/\ka}}_2[C_1+C_2]^{{d/\ka}}.
\]
Now set
\[
\Phi({\omega})=\sup_{x\in\bbR^d}\bigl|\tilde f(x,{\omega})\bigr|.
\]
Then
\[
\bigl|f(x,{\omega})\bigr|\leq\Phi({\omega}) \bigl(1+|x|^{\io+1}\bigr)
\]
and so
\[
\bigl|f\bigl(Z({\omega}),{\omega}\bigr)\bigr|\leq\Phi({\omega})\bigl(1+\bigl|Z({
\omega})\bigr|^{\io+1}\bigr).
\]
These yield (\ref{G}) and (\ref{eq54}) follows by a routine
application of
the H\"{o}lder inequality (see Lemma \ref{L50}).

We now proceed to obtain a H\"{o}lder estimate on $f(x,\omega)$. If
$p(\ka-\delta)>d$, then by (\ref{f4}) and (\ref{f5}) in the same
way as
above for $x,y$ in a cube of side~1,
\[
\bigl|\tilde f(x,\omega)-\tilde f(y,\omega)\bigr|\le C_\delta(\omega)
|x-y|^\delta
\]
with $\|C_\delta(\omega)\|_{p}\le c(\ka, d, \delta)(C_1+C_2)$. For
such a
cube $D$ centered at $z$, we obtain that
\[
\bigl|f(x,\omega)-f(y,\omega)\bigr|\le\tilde C_\delta(z, \omega)
|x-y|^\delta
\]
with
$\|\tilde C_\delta(z, \omega)\|_{p}\le c_7(\ka, d,\delta,\iota)
(1+|z|^{\iota
+1})(C_1+C_2)$. It follows that whenever $|x-y|\le1$,
\[
\bigl|f(x,\omega)-f(y,\omega)\bigr|\le C^\ast(\omega)\bigl[1+|x|^{\iota
+1}+|y|^{\iota+1}
\bigr] |x-y|^\delta,
\]
where $\|C^\ast\|_{p}\le c_8(\ka, d, \delta, \iota) (C_1+C_2)$.
Then for
some $H(\omega)=c_9(\delta,\iota)C^\ast(\omega)$ we obtain the
global estimate
\[
\bigl|f(x,\omega)-f(y,\omega)\bigr|\le H (\omega)\bigl[1+|x|^{\iota+2-\delta
}+|y|^{\iota+2
-\delta}
\bigr] |x-y|^\delta
\]
for all $x,y$. In particular, by Lemma \ref{L50},
\begin{eqnarray*}
&&\bigl\|f(X_1, X_2,\ldots, X_{i-1},
\omega)-f(Y_1, Y_2,\ldots, Y_{i-1},\omega)
\bigr\|_{a} \\
&&\qquad\le\bigl\|H (\omega)\bigl[1+|X|^{\iota+2}
+|Y|^{\iota+2}\bigr] |X-Y|^\delta\bigr\|_{a}\\
&&\qquad\le K   \|H
\|_{p}\bigl(1+\gamma_{m}^{\iota+2}\bigr)\sum
_{j=1}^{i-1}\|X_j-Y_j
\|_{q}^\delta
\end{eqnarray*}
provided $\frac{1}{a}\ge\frac{1}{p}+\frac{\iota+2}{m}+\frac
{\delta}{q}$.
\end{pf*}

In our nonconventional setup Theorem \ref{kol} will be applied in the form
of the following useful result.
%
\begin{corollary}\label{cor}
Let $\cG$ and $\cH_1\subset\cH_2$ be $\sig$-subalgebras on a
probability space
$(\Om,\cF,P)$, $X$ and $Y$ be $d$-dimensional random vectors and $f_i=
f_i(x,{\omega}),  i=1,2$, be collections of random variables that are
continuously
(or separable) dependent on $x\in\bbR^d$ for almost all $\omega$,
measurable with respect to~$\cH_i$,  $i=1,2$, respectively, and satisfy
%
\begin{eqnarray}
\label{f11}
\bigl\|f_i( x,\omega)-f_i( y,\omega)
\bigr\|_{q}&\le& C_1 \bigl(1+|x|^\iota+
|y|^\iota\bigr) |x-y|^\ka
\quad\mbox{and}\nonumber\\[-8pt]\\[-8pt]
\bigl\|f_i(x,\omega)\bigr\|_{q}&\le& C_2
\bigl(1+|x|^\iota\bigr).
\nonumber
\end{eqnarray}
Set $\tilde f_i(x,{\omega})=E[f_i(x,\cdot)|\cG]({\omega})$ and
$g_i(x)=E[f_i(x,{\omega})]$.

\begin{longlist}
\item
Assume that $q\ge p$, $1\ge\ka>\te>\frac dp$ and
$\frac{1}{a}\geq\frac{1}{p}+\frac{\iota+1}{m}$. Then for $i=1,2$,
%
\begin{eqnarray}
\label{Co1}
&&\bigl\| \tilde f_i\bigl(X({\omega}),{\omega}
\bigr)-g_i(X)\bigr\|_a\nonumber\\[-8pt]\\[-8pt]
&&\qquad\leq c \vp_{q,p}(\cG,
\cH_i) (C_1+C_2)^{d/({p\te})}C_2^{1-d/({p\te})}
\bigl(1+\| X\|^{\io+1}_{m}\bigr),\nonumber
\end{eqnarray}
where $c=c(\io,\ka,\te,p,q,a,\del,d)>0$ depends only on the
parameters in
brackets.

\item Next, assume that $\frac{1}{a}\geq\frac{1}{p}+\frac{\iota+2}{m}+
\frac{\del}{q}$. Then for $i=1,2$,
%
\begin{eqnarray}
\label{Co2}
&&\bigl\| E\bigl[f_i(X,\cdot)|\cG\bigr]-g_i(X)
\bigr\|_a\nonumber\\[-8pt]\\[-8pt]
&&\qquad\leq R+2c(C_1+C_2) \bigl(1+2\| X
\|^{\io+2}_{m}\bigr) \bigl\| X-E[X|\cG]\bigr\|^\delta_{q},\nonumber
\end{eqnarray}
where $R$ denotes the right-hand side of (\ref{Co1}).

\item Furthermore, let
$x=(v,z)$ and $X=(V,Z)$, where $V$ and $Z$ are $d_1$ and $d-d_1$-dimensional
random vectors, respectively, and let $f_i(x,{\omega})=f_i(v,z,{\omega
})$ satisfy
(\ref{f11}) in $x=(v,z)$. Set $\tilde g_i(v)=E[f_i(v,Z({\omega
}),{\omega})]$. Then
for $i=1,2$,
%
\begin{eqnarray}
\label{Co3}
&&\bigl\| E\bigl[f_i(V,Z,\cdot)|\cG\bigr]-\tilde
g_i(V)\bigr\|_a\nonumber\\
&&\qquad\leq c \bigl(1+\| X\|^{\io+2}_{m}
\bigr)
\nonumber\\[-8pt]\\[-8pt]
&&\qquad\quad{}\times \bigl(\vp_{q,p}(\cG,\cH_i) (C_1+C_2)^{{d_1}/({p\te})}C_2^{1-
{d_1}/({p\te})}\nonumber\\
&&\qquad\quad\hspace*{20.5pt}{}+
\bigl\| V-E[V|\cG]\bigr\|^\delta_q+\bigl\| Z-E[Z|\cH_i]
\bigr\|^\delta_q \bigr).
\nonumber
\end{eqnarray}

\item Finally, for $a,p,q,\iota,m,\del$ satisfying conditions of \textup{(ii)},
%
\begin{eqnarray}
\label{Co4} &&\bigl\| \tilde f_1\bigl(X({\omega}),{\omega}\bigr)-\tilde
f_2\bigl(Y({\omega}),{\omega }\bigr)-g_1(X)+g_2(Y)
\bigr\|_a
\nonumber\\[-8pt]\\[-8pt]
&&\qquad\leq c \vp_{q,p}(\cG,\cH_2) \bigl(1+\| X
\|^{\io+2}_{m}+\| Y\|^{\io+2}_{m} \bigr)\|
X-Y\|^\del_{q},
\nonumber
\end{eqnarray}
where $c=c(\io,\ka,\te,p,q,a,\del,d)>0$ depends only on the
parameters in
brackets.
\end{longlist}
\end{corollary}
\begin{pf} (i) Set
$h(x,{\omega})=\tilde f_i(x,{\omega})-g_i(x)$,
$K_1=C_1\vp_{q,p}(\cG,\cH_i)$ and $K_2=C_2\vp_{q,p}(\cG,\cH_i)$. Then
by (\ref{f11}) and the definition of $\vp_{q,p}$ for all $x,y\in\bbR^d$
and $q,p\geq1$,
%
\begin{eqnarray}
\label{h1}
&&
\bigl\| h(x,{\omega})-h(y,{\omega})\bigr\|_p \nonumber\\
&&\qquad\leq
\vp_{q,p}(\cG,\cH_i)\bigl\| f_i(x,{
\omega})-f_i(y,{\omega}) -g_i(x)+g_i(y)
\bigr\|_q
\\
&&\qquad\leq 2 K_1\bigl(1+|x|^\io+|y|^{\io}
\bigr)|x-y|^\ka
\nonumber
\end{eqnarray}
and
%
\begin{equation}
\label{h2} \quad\bigl\| h(x,{\omega})\bigr\|_p\leq\vp_{q,p}(\cG,
\cH_i)\bigl\| f_i(x,{\omega })-g_i(x)
\bigr\|_q\leq 2 K_2\bigl(1+|x|^\io\bigr).
\end{equation}
These inequalities enable us to apply Theorem \ref{kol} to
$h(x,{\omega})$ [in
place of $f(x,{\omega})$ there] and (\ref{Co1}) follows from (\ref{eq54}).

(ii) Note that since $ 1> \frac{d}{q}$
it follows that ${\tilde f_i}(x,\omega)$ has an almost surely continuous
modification and taking into account that
$\tilde X=E[X|\cG]$ is $\cG$-measurable, we obtain that
$E[ f_i(\tilde X,\cdot)|\cG]={\tilde f_i} (\tilde X,\cdot)$. Therefore,
%
\begin{eqnarray}
\label{Cop}\quad
&&\bigl\| E \bigl[f_i(X,\cdot)|\cG\bigr]-g_i(X)
\bigr\|_a \nonumber\\
&&\qquad\leq\bigl\| E\bigl[f_i(\tilde X,\cdot )|{\cG}\bigr]-
g_i(\tilde X)\bigr\|_a
\nonumber\\[-8pt]\\[-8pt]
&&\qquad\quad{}+\bigl\|E\bigl[f_i(\tilde X,\cdot)|{\cG}\bigr] -E\bigl[f_i(X,
\cdot)|{\cG}\bigr]\bigr\|_a+ \bigl\| g_i(\tilde X)-
g_i(X)\bigr\|_a
\nonumber
\\
&&\qquad\leq \bigl\| {\tilde f_i}(\tilde X,\cdot)-g_i(\tilde X)
\bigr\|_a +\bigl\|f_i(\tilde X,\cdot)-f_i(X,\cdot)
\bigr\|_a+\bigl\| g_i(\tilde X)-g_i(X)\bigr\|_a.
\nonumber
\end{eqnarray}
We can estimate the first term in the right-hand side of (\ref{Cop}) by
(\ref{Co1}), with $\tilde X$ replacing $X$ and noting that
$\|\tilde X\|_m\le\|X\|_m$. The second term is estimated by (\ref{XY}),
%
\begin{equation}
\label{Cop2} \bigl\|f_i(\tilde X,\omega)-f_i(X,\omega)
\bigr\|_a\le c C_1\bigl(1+\gamma_m^{\iota+2}
\bigr)\| \tilde X-X\|_{q}^\delta.
\end{equation}
The third term is easily estimated taking into account that by (\ref{f11})
and Lem\-ma~\ref{L21},
\[
\bigl|g_i(x)-g_i(y)\bigr|\le c \bigl[1+|x|^\iota+|y|^\iota
\bigr] |x-y|^\ka
\]
and since $0<\delta<\ka\le1$, it follows from H\"{o}lder's
inequality that
\[
\bigl\|g_i(X)-g_i(\tilde X)\bigr\|_a\le c\bigl(1+
\gamma_m^{\iota+2} \bigr)\|\tilde X-X\|_{q}^\delta.
\]

(iii) Set $\tilde V=E[V|\cG]$, $\tilde Z=E[Z|\cH]$,
$\hspace*{1pt}\tilde{\hspace*{-1pt}\tilde g}_i(v)=E[f_i(v,\tilde Z,\cdot)]$ and
$\tilde g_i (v)=\break E[f_i(v,Z,\cdot)]$. Then
%
\begin{eqnarray}
\label{Cop1}
&&
\bigl\| E\bigl[f_i(V,Z,\cdot)|\cG\bigr]-\tilde
g_i(V)\bigr\|_a\nonumber\\
&&\qquad\leq \bigl\| f_i(V,Z,\cdot)-
f_i(V,\tilde Z,\cdot)\bigr\|_a
\\
&&\qquad\quad{} + \bigl\| E\bigl[f_i(V,\tilde Z,\cdot)|\cG\bigr]- \hspace*{1pt}\tilde{\hspace*{-1pt}\tilde
g}_i(V)\bigr\|_a +\bigl\| \hspace*{1pt}\tilde{\hspace*{-1pt}\tilde g}_i(V)-\tilde
g_i(V)\bigr\|_a.
\nonumber
\end{eqnarray}
The first term in the right-hand side of (\ref{Cop1}) is estimated by
(\ref{XY}) similarly to (\ref{Cop2}). Observe that $f_i(v,\tilde
Z,\cdot)$
is $\cH_i$-measurable, and so we can estimate the second term in the
right-hand side of (\ref{Cop1}) by (\ref{Co2}) with $V$, $d_1$,
$\tilde f_i(v,{\omega})$ and $\tilde{\tilde g}_i(v)$ in place of $X$, $d$,
$f_i(x,{\omega})$ and $g_i(x)$, respectively. The third term in the right-hand
side of (\ref{Cop1}) is estimated by first using (\ref{f11}) to obtain
\begin{eqnarray*}
\bigl|\hspace*{1pt}\tilde{\hspace*{-1pt}\tilde g}_i(v)-\tilde g_i (v)\bigr|&\le& E
\bigl[\bigl|f_i(v, \tilde Z, \cdot)- f_i(v, Z,\cdot)\bigr|\bigr]\\
&\le&
E\bigl[\bigl(1+|v|^\iota+ |Z|^\io+|\tilde Z|^\io
\bigr)|Z-\tilde Z|^\ka\bigr]
\end{eqnarray*}
and then substituting $V$ in place of $v$ there.

(iv) Set $\hat h(x,{\omega})=\tilde f_1(x,{\omega})- \tilde
f_2(x,{\omega})-g_1(x)+g_2(x)$, $\hat K_1=C_1\vp_{q,p}(\cG,\cH_2)$ and
$\hat K_2=C_2\vp_{q,p}(\cG,\cH_2)$. Then by (\ref{f11}) and the
definition of $\vp_{q,p}$ for all $x,y\in\bbR^d$ and $q,p\geq1$,
%
\begin{eqnarray}
\label{h3} &&\bigl\|\hat h(x,{\omega})-\hat h(y,{\omega})\bigr\|_p \nonumber\\
&&\qquad\leq
\vp_{q,p}(\cG,\cH_2)\bigl\| f_1(x,{\omega})
-f_1(y,{\omega})
-g_1(x)+g_1(y)\nonumber\\[-8pt]\\[-8pt]
&&\hspace*{60pt}\qquad\quad{}-f_2(x,{
\omega})+f_2(y,{\omega})+g_2(x)-g_2(y)
\bigr\|_q
\nonumber
\\
&&\qquad\leq 2\hat K_1\bigl(1+|x|^\io+|y|^{\io}
\bigr)|x-y|^\ka
\nonumber
\end{eqnarray}
and
%
\begin{eqnarray}
\label{h4} \qquad\bigl\|\hat h(x,{\omega})\bigr\|_p&\leq&\vp_{q,p}(\cG,
\cH_2)\bigl\| f_1(x,{\omega })-f_2(x,{
\omega})-g_1(x) +g_2(x)\bigr\|_q\nonumber\\[-8pt]\\[-8pt]
&\leq&2\hat
K_2\bigl(1+|x|^\io\bigr).\nonumber
\end{eqnarray}
Now (\ref{h3}) and (\ref{h4}) enable us to apply (\ref{XY}), which yields
(\ref{Co4}).
\end{pf}
%
\begin{remark}\label{R41}
We will always work with $a, m, p, \delta, q$ that satisfy
$p(\ka-\delta) >d=(\ell-1)\wp$ and
%
\begin{equation}
\label{eq51} \frac{1}{a}\ge\frac{1}{p}+\frac{\iota+2}{m}+
\frac{\delta}{q}.
\end{equation}
Note also that $m\ge\frac{ap(\iota+2)}{p-a}$ and $m\ge\frac
{aq(\iota+2)} {
q-a\delta}$.
\end{remark}

\section{Limiting covariances}\label{sec4}

Here and in what follows we set
$Y_{i,q_i(n)}=\break  F_i(X(q_1(n)),\ldots, X(q_i(n)))$ and $Y_{i,m}=0$
if $m\ne q_i(n)$ for any $n$. Let
$F_{i,n,r}(x_1,\allowbreak x_2,\ldots,x_{i-1},\omega)=E[F_i(x_1,x_2,\ldots,
x_{i-1},X(n))|\cF_{n-r,n+r}]$ and
$X_r(n)=E[X(n)|\allowbreak\cF_{n-r,n+r}]$. We denote also
$Y_{i,q_i(n),r}=F_{i,q_i(n),r}(X_r(q_1(n)),\ldots, X_r(q_{i-1}(n)),
\omega)$ and $Y_{i,m,r}=0$ if $m\ne q_i(n)$ for any $n$. In view of
(\ref{24}), we can and will always choose continuous in $(x_1,\ldots,x_{i-1})$
versions of conditional expectations $F_{i,n,r}$ which will enable us to
apply Corollary \ref{cor} when needed.

In this section we will study the asymptotical behavior of covariances
\[
D_{i,j}(N,s,t)=E\bigl[\xi_{i,N}(s)\xi_{j,N}(t)
\bigr]=\frac{1}{N}\sum_{1\leq
n\leq Ns} \sum
_{1\leq l\leq Nt} E[Y_{i,q_i(n)}Y_{j,q_j(l)}]
\]
of the processes $\{\xi_{i,N}(t)\}$ defined by (\ref{def21}) and
(\ref{def22}).
We will show that the limits
\[
D_{i,j} (s,t)=\lim_{N\to\infty} D_{i,j}(N,s,t)
\]
exist and $D_{i,j} (s,t) =\min(s,t) D_{i,j}$, where the matrix $\{
D_{i,j}\}$ is
determined by the results below.
%
\begin{proposition}\label{covariance} For any $i,j=1,2,\ldots,k$ and
$s,t>0$ the limit
\begin{eqnarray*}
&&\lim_{N\to\infty}E \bigl[\xi_{i,N}(s)\xi_{j,N}(t)
\bigr]
\\
&&\qquad=\lim_{N\to\infty}\frac{1}{N}\mathop{\sum_{{0\le in\le Ns}}}_{0\le
jl\le Nt} E
\bigl[F_i\bigl(X (n), X (2n),\ldots, X(in)\bigr)\\
&&\hspace*{120.5pt}{} \times F_j
\bigl(X(l), X (2l),\ldots, X (jl)\bigr) \bigr]
\end{eqnarray*}
exists and equals $D_{i,j}\min(s,t)$, which is calculated as follows.
Let ${\upsilon}$ be the greatest common divisor of $i$ and $j$ with
$i={\upsilon}i'$,
$j={\upsilon}j'$ and $i', j'$ being coprime. Set
\begin{eqnarray*}
&&
A_{i,j}(x_{i'}, x_{2i' },\ldots,
x_{{\upsilon}i'},y_{j'}, y_{2j'},\ldots, y_{{\upsilon}j'})\\
&&\qquad=
\int F_i (x_1,\ldots, x_{i-1},x_i)
\\
&&\hspace*{8.3pt}\qquad\quad{}\times F_j (y_1,\ldots, y_{j-1},y_j)
\prod_{\sigma\notin\{i',2i',\ldots,{\upsilon}i'\}} \,d\mu(x_\sigma) \prod
_{\sigma'\notin\{j',2j',\ldots,{\upsilon}j'\}}\,d\mu(y_{\sigma'})
\end{eqnarray*}
and
%
\begin{equation}
\label{eqseven1}\qquad a_{i,j}(n_1,n_2,\ldots,
n_{\upsilon})=\int A_{i,j}(x_1,\ldots,
x_{\upsilon},y_1,\ldots, y_{\upsilon})\prod
_{\sig=1}^{\upsilon}\,d\mu_{n_\sig
}(x_\sig,
y_\sig).
\end{equation}
Then
\[
D_{i,j}=\frac{{\upsilon}}{ij}\sum_{u=-\infty}^\infty
a_{i,j}(u, 2u,\ldots, {\upsilon}u),
\]
where
\[
a_{i,j}(0, 0,\ldots, 0)=\int A_{i,j}(x_1,\ldots, x_{\upsilon
},x_1,\ldots, x_{\upsilon}) \prod
_{\sigma=1}^{\upsilon}d\mu(x_\sigma)
\]
and the series for $D_{i,j}$ converges absolutely.
\end{proposition}

This is essentially a straightforward but long computation carried out in
a few steps, each one formulated as a lemma. We will first derive
some
uniform bounds on $D_{i,i}(N,t,t)$.
A key step is to get for any pair $i,j$ an estimate on
\[
b_{i,j}(n,l)=E[Y_{i,q_i(n)}Y_{j,q_j(l)}].
\]
If $ |n-l|\gg 1$, then either $q_i(n)$ or $q_j(l)$ will be much bigger than
all other $q_i(m)$ and $q_j(m)$, which together with the mean $0$
condition on
$F_i, F_j$ and estimates of Section \ref{sec3} will make then this expectation
small, as shown in the following result which will also be used later on.
%
\begin{lemma}\label{lem52} There exists a nonincreasing function
$h(m)\ge0$,
with\break  $\sum_{m=1}^\infty h(m)<\infty$, such that for any
$i,j=1,2,\ldots,\ell$,
%
\begin{equation}
\label{varbound0} \sup_{n,l\dvtx   s_{i,j}(n,l)\geq m}\bigl|b_{i,j}(n,l)\bigr|\le h(m),
\end{equation}
where $s_{i,j}(n,l)=\max(\hat s_{i,j}(n,l),\hat s_{j,i}(l,n))$ and
$\hat s_{i,j}(n,l)=\min(q_i(n)-q_j(l),  n)$.
Furthermore, there exists a constant $C>0$ such that for all $t\geq
s\geq0$
and $i=1,\ldots,\ell$,
%
\begin{equation}
\label{varbound} \sup_{N\ge1} E\bigl|\xi_{i,N}(t)-
\xi_{i,N}(s)\bigr|^2\le C (t-s).
\end{equation}
\end{lemma}
\begin{pf}
First, observe that for $i=1,\ldots,k$,
%
\begin{equation}
\label{s-est1} q_i(n)-q_{i-1}(n)=n \quad\mbox{and}\quad
s_{i,i}(n,l)=\min \bigl(i|n-l|,\max(n,l)\bigr) \geq|n-l|,\hspace*{-30pt}
\end{equation}
where in the first equality we set $q_0(n)=0$. On the other hand, if
$i\geq
k+1$, then it follows from (\ref{2q1})--(\ref{2q3}) that for any
$\ve>0$
there exists $n_\ve$ such that for all $n\geq n_\ve$ and $n>l\geq0$,
%
\begin{equation}
\label{s-est2} q_i(n)-q_{i-1}(n)\geq n+\ve^{-1},\qquad
q_i(n)-q_i(l)\geq n-l+\ve^{-1}
\end{equation}
and so
%
\begin{equation}
\label{s-est3} s_{i,i}(n,l)\geq\min\bigl(n-l+\ve^{-1},n\bigr)
\geq n-l.
\end{equation}

Now, assume that $q_i(n)-q_j(l)\geq0$ and $n\geq n_1$ so that we will use
here (\ref{s-est1})--(\ref{s-est3}) with $\ve=1$, while only in Proposition
\ref{covariance2} these\vspace*{1pt} estimates will be needed for all positive $\ve$.
Set $r=\frac13s_{i,j}(n,l)=\frac13\hat s_{i,j}(n,l)$.
If we replace $Y_{i,q_i(n)}$ and
$Y_{j,q_j(l)}$ by $Y_{i,q_i(n),r}$ and $Y_{j,q_j(l),r}$ defined at the
beginning of this section, then the difference between $b_{i,j}(n,l)$ and
\[
b^{(r)}_{i,j}(n,l)=E[ Y_{i,q_i(n),r}Y_{j,q_j(l),r}]
\]
can be estimated easily using Corollary \ref{cor}(iv) with $\cH_2=\cF$,
which gives
\[
\bigl|b^{(r)}_{i,j}(n,l)-b_{i,j}(n,l)\bigr|\le c(
\gamma_m, \gamma_{
{2p(\iota+1)}
/({2- p\alpha})}) \bigl[\beta(q,r)
\bigr]^\delta.
\]
On the other hand, by (\ref{s-est1}) and (\ref{s-est2}) we see that
in our
circumstances $\min(q_i(n)-q_j(l),q_i(n)-q_{i-1}(n))\geq\hat s_{i,j}(n,l)$,
and so by Corollary \ref{cor}(i),
\begin{eqnarray*}
\bigl|b^{(r)}_{i,j}(n,l)\bigr|&=&|EY_{i,q_i(n),r}Y_{j,q_j(l),r}|
\\
&=& \bigl\vert E \bigl[E [Y_{i,q_i(n),r}|\cF_{0,q_i(n)-r} ]Y_{j,q_j(l),r}
\bigr] \bigr\vert\\
&\le&\bigl\|F_j \bigl(X_r\bigl(q_1(l)
\bigr),\ldots, X_r\bigl(q_j(l)\bigr)\bigr)
\bigr\|_{L_2(P)}
\\
&&{}\times\bigl\|E[Y_{i,q_i(n),r}|\cF_{0,q_i(n)-r} ]\bigr\|_{L_2(P)} \\
&\le& C
\varpi_{q,p}\bigl(\tfrac13s_{i,j}(n,l)\bigr).
\end{eqnarray*}
We can always estimate $|b_{i,j}(n,l)|$ by $|b^{(r)}_{i,j}(n,l)-b_{i,j}(n,l)|
+|b^{(r)}_{i,j}(n,l)|$, so that
\[
\bigl|b_{i,j}(n,l)\bigr|\le C \bigl(\varpi_{q,p}\bigl(
\tfrac13s_{i,j}(n,l)\bigr)+ \bigl[\beta\bigl(q,\tfrac13s_{i,j}(n,l)
\bigr)\bigr)^\delta \bigr).
\]
Now, observe that if $n<n_1$ and $q_i(n)-q_j(l)\geq0$, then
\[
s_{i,j}(n,l)\leq L_1=\max_{n<n_1,  i\leq\ell}q_i(n)
\quad \mbox{and}\quad    l\leq n_1+L_1.
\]
Hence, in order to satisfy (\ref{varbound0}), we can take
\[
h(m)=\max_{0\leq n,l\leq n_1+L_1,  1\leq i,j\leq\ell}\bigl|b_{i,j}(n,l)\bigr|
\]
for $m\leq L_1$, while for $m>L_1$ we define
\[
h(m)=C \bigl(\varpi_{q,p}\bigl(\bigl[\tfrac13m\bigr]\bigr)+\bigl(
\beta\bigl(q,\bigl[\tfrac13m\bigr]\bigr)\bigr)^\delta \bigr).
\]

Finally, by (\ref{s-est1}) and (\ref{s-est3}) for $t\geq s\geq0$,
\begin{eqnarray*}
E\bigl[\bigl|\xi_{i,N}(t)-\xi_{i,N}(s)\bigr|^2\bigr]&\le&
\frac{1}{N} \biggl( \sum_{Ns\leq l\le Nt}
b_{i,i}(l,l)+2\mathop{\sum_{Ns\leq l\le Nt}}_{n\ge l+1}\bigl|b_{i,i}(n,l)\bigr|
\biggr)
\\
&\le&\frac{1}{N} \sum_{Ns\leq l\le Nt}
\biggl(EY_{i,l}^2+ 2\sum_{n\geq l+1}
h(n-l) \biggr)\le C t
\end{eqnarray*}
provided $N(t-s)\geq1$, and the result follows.
\end{pf}

Next, we will need a result which will be formulated in a somewhat more
general situation.
Let $H(x_1,x_2, \ldots, x_d)$ be a function on $(R^\nu)^d$ that is continuous
and satisfies the growth condition $|H(x_1,x_2,\ldots, x_d)|\le
1+\sum_i\|x_i\|^\iota$ for some $\iota\ge1$. Suppose that
$\{Y(n)\dvtx n\ge1\}$ is a stochastic process with values in $R^\nu$
and there exists an integer $m\geq1$ such that for any $l\le m$ the
distribution of $\{{Y(n_1),Y(n_2),\ldots, Y(n_{l})}\}$ depends only on
the spacings $\{n_i-n_{i-1}\},  l=2,\ldots,l$ between them. For $l\ge2$, we
denote this distribution by $\mu_{S}$,
where $S$ is a set of $l-1$ positive integers prescribing the spacings
between the $l$ integers. We assume that all $\{Y(n),  n\geq1\}$ have
a common distribution $\mu$ and that the integrability condition
$\int\|x\|^\iota \,d\mu<\infty$ holds true.
For some $p,q\geq1$ and a nested family of sub $\sigma$-fields $\cF_{m,n}$
as above assume the mixing condition
\[
\varpi_{q,p}(l)=\sup_{m-n\ge l} \varpi_{q,p}(
\cF_{-\infty,m},\cF_{n,\infty}) \to0 \qquad\mbox{as } l\to\infty
\]
and the localization condition
\[
\lim_{r\to\infty} \sup_n\bigl\|Y(n)-E\bigl[Y(n)|{\cF}_{n-r,n+r}
\bigr]\bigr\|_{L_1(P)}=0.
\]
Let $n_1<n_2<\cdots<n_d$ be a sequence of integers that tend to
$\infty$
with some of the gaps $\{n_{i+1}-n_i\}$ tending to infinity while
others are
kept fixed. This splits the set of integers $1,2,\ldots, d$ into a
partition $\cP$ consisting of blocks $B_j$ of different sizes. The
pairwise distances between integers in each block $B_j$ remain fixed
(so it can be viewed as rigid), while the distances between different
blocks tend to $\infty$. We assume that each block $B_j$ consists of
at most
$m$ integers. Let $m_j$ denote the number of integers in a block $B_j$
and $S_j$ denote the set of spacings in $B_j$, that is, the sequence of $m_j-1$
positive integers representing pairwise distances between successive integers
in $S_j$. Let the distribution $\mu_{\cP}$ on $(R^l)^d$ be the
product measure
\[
\mu_{\cP}=\Pi_j \mu_{S_j}
\]
over successive blocks.
%
\begin{lemma}\label{limcov} Assume that $\{n_j\}$ goes to infinity with
rigid blocks determined by $\cP$. Then
\[
\lim_{n_1,\ldots, n_d\to\infty}E\bigl[H\bigl(X(n_1),\ldots,
X(n_d)\bigr)\bigr]=\int H(x_1,\ldots,
x_d)\,d\mu_{\cP},
\]
where the limit is taken so that the sets $S_j$ of spacings in each block
$B_j$ remain fixed while the gaps between different blocks tend to infinity.
\end{lemma}
\begin{pf}
First we note that because of the growth and integrability conditions we
can replace $H$ by $H\phi$, where $\phi$ is a continuous cutoff function
with compact support. The error is uniformly controlled on either side. We
can then approximate $H$ uniformly by a smooth function. In other
words, we
can assume without loss of generality that $H$ is a bounded continuous
function supported on some ball of radius $L$ with a bounded gradient.
We prove the lemma by reducing the number of blocks by one at each step.
The last gap that tends to $\infty$ cuts off a block $B=\{n_{d'+1},\ldots, n_d\} $ at the end with a rigid spacing $S$ between integers in
the block. We will show that
%
\begin{equation}
\label{reduction}
\mathop{\lim_{n_1,\ldots, n_d\to\infty}}_{\cP\ \mathrm{fixed}}E\bigl[{\hat H}\bigl(X(n_1),\ldots, X(n_d)\bigr)\bigr]=0,
\end{equation}
where
\begin{eqnarray*}
{\hat H}(x_1,x_2, \ldots, x_d)&=&H(x_1,x_2,
\ldots, x_d)
\\
&&{}-\int H(x_1,x_2, \ldots, x_{d'},x_{d'+1},\ldots, x_d) \,d\mu_S(x_{d'+1},\ldots,x_d).
\end{eqnarray*}
This will reduce the number of blocks by one, replacing $H$ by
\[
H_1(x_1,x_2, \ldots, x_{d'})=
\int H(x_1,x_2, \ldots, x_{d'},
x_{d'+1},\ldots, x_d ) \,d\mu_S(x_{d'+1},\ldots, x_d).
\]
The step by step reduction will end when only the first block $B_1$
with spacings $S_1$ remains and since it is rigid, we can integrate it out
with $\mu_{S_1}$ and end up with $\int H(x_1,\ldots, x_d)\,d\mu_{\cP}$,
which will complete the proof of the lemma.

The function $\hat H$ is also bounded with a bounded gradient.
Therefore,
\begin{eqnarray*}
&&
\bigl\|{\hat H}\bigl(X(n_1),\ldots, X(n_d)\bigr)-{
\hat H}\bigl(X_r(n_1),\ldots,X_r(n_d)
\bigr)\bigr\|
\\
&&\qquad\le C\sup_n\bigl\|X_r(n)-X(n)\bigr\|_{L_1(P)}\to0
\end{eqnarray*}
uniformly over all $n_1,\ldots, n_d$ as $r\to\infty$. To establish
(\ref{reduction}), it is therefore sufficient to prove that
%
\begin{equation}
\label{reduction2} \lim_{r\to\infty} \limsup_{n_1,\ldots, n_d\to\infty} E\bigl[{\hat
H}\bigl(X_r(n_1),\ldots, X_r(n_d)
\bigr)\bigr]=0.
\end{equation}
Observe that
\begin{eqnarray*}
E\bigl[{\hat H}\bigl(X_r(n_1),\ldots,
X_r(n_d)\bigr)\bigr]&=&E\bigl[E\bigl[{\hat H}
\bigl(X_r(n_1),\ldots, X_r(n_d)
\bigr)|\cF_{-\infty,n_{d'}+r}\bigr]\bigr]
\\
&=&E\bigl[G_r\bigl(X_r(n_1),\ldots,
X_r(n_{d'}),\omega\bigr)\bigr],
\end{eqnarray*}
where
\[
G_r(x_1,\ldots, x_{d'},\omega)=E\bigl[{
\hat H}\bigl(x_1,\ldots, x_{d'}, X_r(n_{d'+1}),\ldots,
X_r(n_d)\bigr)|\cF_{-\infty,n_{d'}+r}\bigr].
\]
To prove (\ref{reduction2}), it is clearly sufficient to show that
\[
\lim_{r\to\infty}E\Bigl[\sup_{x_1,\ldots, x_{d'}}\bigl|G_r(x_1,\ldots, x_{d'}, \omega)\bigr|\Bigr]=0.
\]
Since $\|\nabla_x G_r\|_\infty\le\|\nabla_x{\hat H}\|_\infty
\le
\|\nabla_x H\|_\infty$, there is a uniform bound on $\|\nabla G_r\| $.
We can therefore estimate
\[
\sup_{x_1,\ldots, x_{d'}}\bigl|G_r(x_1,\ldots, x_{d'},
\omega)\bigr|\le C\int\bigl|G_r(x_1,\ldots, x_{d'},
\omega)\bigr| \,dx_{1}\cdots dx_{d'}.
\]
Taking expectations and observing that $G_r$ vanishes outside a ball of
radius~$L$,
\[
E\sup_{x_1,\ldots, x_{d'}}\bigl|G_r(x_1,\ldots,
x_{d'},\omega)\bigr|\le C L^{d'} \sup_{x_1,\ldots, x_{d'}}E\bigl|G_r(x_1,\ldots, x_{d'},\omega)\bigr|.
\]
If $n_{d'+1}-n_{d'}>2r$, then by the definition (\ref{eq21}) of the
dependence coefficients~$\varpi$,
\begin{eqnarray*}
&&\sup_{x_1,\ldots, x_{d'}}\bigl\|G_r(x_1,\ldots,
x_{d'},\omega)-{\hat H}_r(x_1,\ldots,
x_{d'})\bigr\|_1\\
&&\qquad\le2 \varpi_{\infty,1}(n_{d'+1}-n_{d'}-2r)
\|H\|_\infty,
\end{eqnarray*}
where
\[
{\hat H}_r (x_1,\ldots, x_{d'})=E\bigl[{
\hat H}\bigl(x_1,\ldots, x_{d'}, X_r(n_{d'+1}),\ldots,
X_r(n_d)\bigr)\bigr],
\]
while
\[
{\hat H} (x_1,\ldots, x_{d'})=E\bigl[{\hat H}
\bigl(x_1,\ldots, x_{d'}, X(n_{d'+1}),\ldots,
X(n_d)\bigr)\bigr]\equiv0.
\]
Since $\hat H$ has a bounded gradient,
\begin{eqnarray*}
&& \bigl|E\bigl[{\hat H}\bigl(x_1,\ldots, x_{d'},
X(n_{d'+1}),\ldots, X(n_d)\bigr)\bigr]\\
&&\quad\hspace*{0pt}{} - E\bigl[{\hat H}
\bigl(x_1,\ldots, x_{d'},
X_r(n_{d'+1}),\ldots, X_r(n_d)
\bigr)\bigr]\bigr|\\
&&\qquad\le C\sup_n E\bigl|X(n)-X_r(n)\bigr| =\epsilon(r)
\to0 \qquad\mbox{as } r\to\infty.
\end{eqnarray*}
Taking into account that $\vp_{\infty,1}(l)\leq\vp_{p,q}(l)\to0$, the
lemma follows from the above estimates.
\end{pf}
%
\begin{lemma}\label{Lb2} For any $i,j\leq k$ and $s,t>0$ and integer
$u$, the
limit
%
\begin{equation}
\label{sum1} \lim_{N\to\infty}\frac{1}{N}
\mathop{\mathop{\sum_{0\le in \le Ns}}_{0\le jl \le
Nt}}_{
in-jl= u} b_{i,j}(n,l)=\frac{{\upsilon}\min(s,t)}{ij}
 c_{i,j}(u)
\end{equation}
exists where ${\upsilon}$ is the greatest common divisor of $i$
and $j$. For any multiple of ${\upsilon}$,
%
\begin{equation}
\label{c} c_{i,j}({\upsilon}u)=a_{i,j}(u,2u,\ldots,{\upsilon}u)
\end{equation}
with $a_{i,j}$ defined by (\ref{eqseven1}). If $u$ is not a multiple
of ${\upsilon}$,
then $c_{i,j}(u)=0$. Furthermore,
%
\begin{equation}
\label{sum2} \lim_{N\to\infty}\frac{1}{N}
\mathop{\sum_{0\le in\le Ns}}_{0\le jl\le Nt} b_{i,j}(n,l)=\frac{{\upsilon}\min(s,t)}{ij}\sum
_{-\infty<u <\infty
} c_{i,j}(u)
\end{equation}
and the series in the right-hand side converges absolutely.
\end{lemma}
\begin{pf}
It is clear that if $u$ is not a multiple of ${\upsilon}$, there are
no solutions
of the equation
$in'-jl'=u$, so we can replace $u$ by ${\upsilon}u$. Combining the
indices $n, 2n,\ldots, in$ and $l,2l,\ldots,jl$\vadjust{\goodbreak} and ordering them into a single sequence,
we obtain by employing Lemma \ref{limcov} that
\begin{eqnarray*}
&&
\mathop{\lim_{n,l\to\infty}}_{in-jl={\upsilon}u}b_{i,j}(n,l)\\
&&\qquad=
\mathop{\lim_{n,l\to\infty}}_{in-jl={\upsilon}u}E
\bigl[F_i\bigl(X(n), X(2n),\ldots,X(in)\bigr)
F_j\bigl(X(l),X(2l),\ldots, X(jl)\bigr)
\bigr]\\
&&\qquad=a_{i,j}(u,2u,\ldots,{\upsilon}u).
\end{eqnarray*}
If ${\upsilon}$ is the greatest common divisor of
$i$ and $j$, then $i={\upsilon}\alpha$ and $j={\upsilon}\beta$ with
$\alpha$ and
$\beta$ being coprime. Since all the gaps in either sequence above go to
$\infty$, we can have blocks of size more than one only by pairing two
members from different sequences and, therefore, the rigid blocks of
Lemma \ref{limcov} can be of
size one and two only. If we start with $(n,l)$ such that $\alpha n-
\beta l=u$, their multiples $(\alpha m n,\beta m l)$,
$m=1,\ldots,{\upsilon}$, with
$\alpha m n-\beta m l =mu$ will give ${\upsilon}$ blocks of size $2$.
There cannot be any other. Indeed, if $(a,b)$ is a pair of integers which is not an
integer multiple of $(\alpha, \beta)$, then taking into account that
$\al$ and
$\be$ are coprimes, we conclude that $|an-bl|\to\infty$ when $n\to
\infty$,
preserving $\al n-\be l=u$ fixed. To complete the proof of
the lemma, we need to count the number of integer solutions of $in-jl=
{\upsilon}u$ or $\alpha n-\beta l=u$ with $\alpha{\upsilon}n\le Nt$
and $\beta
{\upsilon}l\le Ns$. The set of solutions for any $u$ is obtained by shifting
the set of solutions of the homogeneous equation $\alpha n-\beta l=0$
by a fixed solution of the above nonhomogeneous one. Therefore, with
our constraints their numbers can differ at most by a constant. In
the homogeneous case
the solutions are precisely those $m=in=jl$ that are multiples of
${\upsilon}\alpha\beta$. Their number is an integral value of $\frac{N
\min\{t,s\}}{{\upsilon}\alpha\beta}=\frac{N {\upsilon}\min\{s,t\}
}{ ij}$.
This\vspace*{1pt} proves (\ref{sum1}), while Lemma \ref{lem52} and (\ref{sum1})
imply (\ref{sum2}).
\end{pf}

Finally, we turn to $\xi_{i,N}(t)$ with $k+1\le i\le\ell$. We will see
in the next section that, in fact, their limits in distribution
$\{\eta_i(\cdot); i\ge k+1\}$ are mutually independent processes which
are also independent of the processes $\{\eta_i(\cdot); 1\le i\le k\}$,
but here we deal only with their variances and covariances.
%
\begin{proposition}\label{covariance2} For $i\ge k+1$,
%
\begin{eqnarray}
\label{cov,ii}
&&
\lim_{N\to\infty}E \bigl(\xi_{i,N}(s)
\xi_{i,N}(t) \bigr)
\nonumber\\[-8pt]\\[-8pt]
&&\qquad=\min(s,t)\int \bigl(F_i(x_1, x_2,\ldots,
x_i) \bigr)^2 \,d\mu(x_1)\,d\mu(x_2)
\cdots d\mu(x_i).
\nonumber
\end{eqnarray}
Moreover, for any $t,s$ and $j<i$, $i>k$,
%
\begin{equation}
\label{cov,ij} \lim_{N\to\infty}E \bigl(\xi_{i,N}(t)
\xi_{j,N}(s) \bigr)=0.
\end{equation}
\end{proposition}
\begin{pf}
It follows from (\ref{s-est3}) that
\[
s_{i,i}(n,l)\geq\min\bigl(|n-l|+\ve^{-1},\max(n,l)\bigr)
\qquad\mbox{if }\max(n,l) \geq n_\ve  \mbox{ and } n\ne l
\]
and so, by (\ref{varbound0}),
\[
b_{i,i}(n,l)\to0   \qquad\mbox{as }   \max(n,l)\to\infty\qquad
\mbox{so that } |n-l|\geq1.
\]
Therefore, for any fixed $L\geq n_1$,
\begin{eqnarray*}
&&
\limsup_{N\to\infty}\frac1N \sum_{1\leq n,l\leq TN,n\ne
l}\bigl|b_{i,i}(n,l)\bigr|\\
&&\qquad\le 2T\sum_{m\geq L}h(m)
+\limsup_{N\to\infty}\frac1N\mathop{\sum_{1\le|n-l|\le L}}_{n,l\le TN}\bigl|b_{i,i}
(n,l)\bigr|\\
&&\qquad= 2T\sum_{m\geq L}h(m).
\end{eqnarray*}
We now let $L\to\infty$ and since $\sum_m h(m)<\infty$, it follows that
$\limsup$ in the left-hand side above equals zero, that is, the off-diagonal
terms do not contribute in (\ref{cov,ii}). It remains to deal with
the diagonal terms $b_{i,i}(n,n)$. Since $q_j(n)-q_{j-1}(n)\to\infty$
for $j=2,3,\ldots,\ell$ as $n\to\infty$, it follows from Lem\-ma~\ref{limcov}
that
%
\begin{equation}
\label{xi5} \lim_{n \to\infty}b_{i,i}(n,n)=\int \bigl(F_i(x_1,\ldots,x_i)
\bigr)^2\,d\mu(x_1)\cdots d\mu(x_i),
\end{equation}
proving (\ref{cov,ii}).

Next, we deal with (\ref{cov,ij}). Relying on Lemma \ref{lem52}, we can
estimate for any $\ve>0$,
%
\begin{eqnarray}
\label{crosscov}\quad &&\bigl|E\xi_{i,N}(t)\xi_{j,N}(s)\bigr|
\nonumber\\
&&\qquad\leq\bigl|E\xi_{i,N}(\ve T)\xi_{j,N}(s)\bigr| +\bigl|E \bigl(
\xi_{i,N}(t)-\xi_{i,N}(\ve T) \bigr)\xi_{j,N}(s)\bigr|
\nonumber\\[-8pt]\\[-8pt]
&&\qquad\leq \bigl(E\xi^2_{i,N}(\ve T) \bigr)^{1/2}
\bigl( E\xi^2_{j,N}(s) \bigr)^{1/2} +\frac1N\sum
_{\ve NT\leq n\leq NT,1\leq l\leq
NT}\bigl|b_{i,j}(n,l)\bigr|
\nonumber
\\
&&\qquad\le CT\sqrt{\ve} +\frac1N\sum_{\ve NT\leq n\leq NT,1\leq l\leq NT} h
\bigl(s_{i,j}(n,l)\bigr).
\nonumber
\end{eqnarray}
Since $i>j$ and $i>k$, then, by (\ref{2q3}), we can choose
$N(\ve)>\ve^{-1}T^{-1}n_\ve$ such
that $q_i(n)-q_j(l)>\ve^{-1}$ whenever $N\geq N(\ve),  n\geq\ve
NT,
l\leq NT$ and, moreover, by (\ref{s-est2}),
\begin{eqnarray*}
s_{i,j}(n,l)&=&\min\bigl(q_i(n)-q_j(l),n
\bigr)
\\
&\geq&\min\bigl(q_i(n)-q_i(\ve NT)+\ve^{-1},n
\bigr)\geq\min\bigl(n-\ve NT+\ve^{-1},n\bigr).
\end{eqnarray*}
Hence,
\[
\frac1N\sum_{\ve NT\leq n\leq NT,1\leq l\leq NT}h\bigl(s_{i,j}(n,l)
\bigr)\leq T\sum_{m\geq\min(\ve^{-1},\ve NT)}h(m)
\]
and letting, first, $N\to\infty$ and then $\ve\to0$, we derive
(\ref{cov,ij})
from (\ref{crosscov}).
\end{pf}

\section{Proof of the main theorem}\label{sec5}

The proof of Theorem \ref{MainThm} relies on martingale approximations and
martingale limit theorems, but we will need several modifications in our
situation.
We begin with the following result which can be found in various forms in
the literature (see, e.g., Section 2 in Chapter VIII of \cite
{JS} and
close versions in Theorem 18.2 in \cite{Bi} and Theorem 4.1 in \cite{HH}).
For each $N$ let $\cG_{N,n},  n=1,2,\ldots\,$, be a filtration of $\sigma$-algebras
and let $\{U_{N,n}\dvtx  n\ge1\}$ be a triangular array of random variables
satisfying the following conditions:

B1. For every $N$, $\{U_{N,n}\}$ is adapted to some $(\Omega_N,
\cG_{N,n}, P_N),  n=1,2,\ldots\,$;

B2. $\{U_{N,n}\}$ are uniformly square integrable;

B3. $\|E[U_{N,m}|\cG_{N, n}]\|_2\le c(m-n)$ for all $N$, $n\le
m$ and
for some sequence $c(k)$ satisfying $\sum_{k=0}^\infty c(k)=C<\infty$;

B4. For some increasing function $A(t)$,
\[
\lim_{N\to\infty} \biggl\|\frac{1}{N}\sum_{1\le n\le Nt}
W_{N,n}^2-A(t)\biggr\|_{L_1(P)} =0,
\]
where
\[
W_{N,n}=U_{N,n}+ \sum_{m\ge n+1}E[U_{N,m}|
\cG_{N,n}]- \sum_{m\ge n}E[U_{N,m} |
\cG_{N,n-1}].
\]

Observe that $W_{N,n},  n\geq1$ is a martingale differences sequence
provided B1--B3 hold true.
%
\begin{theorem}\label{TriCLT} Under assumptions \textup{B1--B4},
\[
\xi_N(t)=\frac{1}{\sqrt N}\sum_{1\le n\le Nt}U_{N,n}
\]
converges in distribution on $D[[0,T]; R]$ to a Gaussian process $\xi
(t)$ with
independent increments such that $\xi(t)-\xi(s)$ has mean $0$ and
variance \mbox{$A(t)-A(s)$}.
\end{theorem}

We need, however, to strengthen the theorem a little bit in our
context. First
we note that the condition B4 can be replaced by the weaker condition
%
\begin{equation}
\label{b4=b5} \lim_{N\to\infty} \frac{1}{N}\sum
_{1\le n\le Nt} E\bigl[W_{N,n}^2\bigr]=A(t)
\end{equation}
as can be seen from the following result.
%
\begin{lemma}
If for a fixed $l$ the random variables
\[
V_{N,r}=\biggl(\sum_{r(l-1)+1\le
n\le rl} U_{N,n}\biggr)^2
\]
satisfy a uniform law of large numbers in the
sense that
\[
\lim_{r\to\infty}\sup_N \sup_n E \Biggl[\Biggl|
\frac{1}{r} \sum_{j=1}^r \bigl[
V_{N,n+j} -E[V_{N,n+j}] \bigr]\Biggr| \Biggr]=0,
\]
then (\ref{b4=b5}) implies \textup{B4}.
\end{lemma}
\begin{pf}
We begin with the observation that if $\eta_n,  n\geq1$ are martingale
differences adapted to any filtration ${\cG}_n$
and they are uniformly integrable, then $\frac{1}{N}\sum_{n=1}^N \eta_n\to0$
in $L_1(P)$. To see this, we approximate $\eta_n$
in $L_1(P)$ by ${\tilde\eta}_n$ that are uniformly bounded. The latter
may not be a martingale difference, but it can be written as
${\tilde
\eta}_n={\hat\eta}_n+ {\bar\eta}_n$ with $\|{\bar\eta}_n\|_{L_1(P)}
\le\| \eta_n- {\tilde \eta}_n\|_{L_1(P)}$ and ${\hat\eta
}_n$ being
a martingale difference with a uniformly bounded second moment.

We will now compare
\[
A_N(t,\omega)=\frac{1}{N}\sum_{n\le[Nt]}
(\eta_n)^2
\]
with block sums over $B_r=\{n\dvtx  rl +1\le n\le(r+1)l\}$,
\[
A^l_N(t,\omega)=\frac{1}{N}\sum
_{r\dvtx  B_r\subset[0,Nt]} \biggl(\sum_{n\in B_r}
\eta_n\biggr)^2.
\]
The difference involves the cross terms
\[
A^l_N(t,\omega)-A_N(t,\omega)=
\frac{2}{N}\sum_{{r: B_r\subset[0,Nt]}}
\mathop{\sum_{n>m}}_{n,m\in B_r} \eta_n\eta_m.
\]
It is easy to see that the sum
\[
\mathop{\sum_{n>m}}_{n,m\in B_r} \eta_n\eta_m
\]
is a martingale difference (in $r$) adapted to $\cG_{rl}$ and, therefore, for fixed
$l$,
\[
\lim_{N\to\infty}\bigl\|A^l_N(t,\omega)-A_N(t,
\omega)\bigr\|_{L_1(P)}=0.
\]
Since $E^P[A_N(t,\omega)]=E^P[A^l_N(t,\omega)]$, it follows immediately
that
\begin{eqnarray*}
&&
\limsup_{N\to\infty}\bigl\|A_N(t,\omega)-E^P
\bigl[A_N(t,\omega)\bigr]\bigr\|_{L_1(P)}
\\
&&\qquad\le\limsup_{N\to\infty}\bigl\|A^l_N(t,
\omega)-E^P\bigl[A_N(t,\omega)\bigr]\bigr\|_{L_1(P)}.
\end{eqnarray*}
On the other hand, $W_{N,n}=U_{N,n}-R_{N,n-1}+R_{N,n}$, where
\[
R_{N,n}=\sum_{m\ge n+1}E[U_{N,m}|\cG_{N,n}]
\]
and
\[
\sum_{ n\in B_r} W_{N,n}=\sum
_{n\in B_r}U_{N,n} -R_{N, jl}+ R_{N,(j+1)l}.
\]
By our assumption, the squares of the block sums $V_{N,r}=(\sum_{n\in B_r}
U_{N,n})^2$ satisfy a uniform law of large numbers in $L_1(P)$. The
differences
between the two block sums come from the correction term and their second
moments are uniformly controlled. Therefore, their contribution is at most
$\frac{C}{l}$. Hence,
\[
\limsup_{l\to\infty} \limsup_{N\to\infty}\bigl\|A^l_N(t,
\omega)-E^P\bigl[A_N(t, \omega)\bigr]\bigr\|_{L_1(P)}=0
\]
and the lemma follows.
\end{pf}
%
\begin{remark}
Let the filtration $\cF_{m,n}$ satisfy any mixing condition, that is,
$\varpi_{p,q}(k)\to0$ as $k\to\infty$.
Then any collection of uniformly integrable random
variables $\{f_n(\omega)\}$, with $f_n$ being $\cF_{n+k,n-k}$
measurable for
some fixed $k$, are easily seen to satisfy the (centered) law of large
numbers. It is obvious
for uniformly bounded $\{f_n\}$ and we can always approximate our $\{
f_n\}$
uniformly in $L_1$ by uniformly bounded ones.
\end{remark}
%
\begin{corollary}
If we have a family of triangular arrays and the conditions of Theorem
\ref{TriCLT} are valid uniformly over the family, then the limit theorem
is also valid uniformly over the family.
\end{corollary}
\begin{pf} The proof is a routine argument by contradiction. If the
family is indexed by $\alpha$ and the limit theorem is not valid uniformly,
then for some choice $\alpha_N$ that depends on $N$ the limit theorem
fails to hold. But this is just another triangular array and, by the
uniform validity of the assumptions, the limit theorem has to hold.
\end{pf}
%
\begin{remark}\label{conditional}
For each $N$ let $\cG_{N,n},  n=1,2,\ldots\,$, be a filtration of $\sig$-algebras
and let $k_N\geq1,  N=1,2,\ldots\,$, be an integer sequence with $k_N\to
\infty$
as $N\to\infty$. One way to generate new triangular arrays for
$N=1,2,\ldots\,$,
is to take a sequence of sub
$\sigma$-fields, ${\cG}_{ N,k_N}$, a sequence of sets $B_N\in{\cG}_{N,k_N}$
with $P_N(B_N)\ge\delta>0$ and to consider $(\Omega_N,\tilde{ \cG}_{N,n},
\tilde{U}_{N,n}, P_{N,B_N}),  n=1,2,\ldots\,$, where $\tilde{ \cG}_{N,n}=
{\cG}_{N,k_N+n}$, $\tilde{U}_{N,n}= U_{N,k_N+n}$ and the measure
$P_{N, B_N}$
is defined by
\[
P_{N,B_N}(\Gam)=\frac{P_N(\Gam\cap B_N)}{P_N(B_N)}.
\]
It is easy to see that $\tilde U_n$ are again martingale differences,
for each fixed $\delta>0$ uniform integrability under
$P_{N,B_N}$ is inherited from the same property under $P_N$ and the
condition B3 of Theorem \ref{TriCLT} holds uniformly over this family
as well, provided $k_N\le CN$ for some $C$. Otherwise, it has to be
checked again. The limit $A(t)$ will of course vary
depending on the behavior of $\frac{k_N}{N}$. If $\frac{k_N}{N}\to t_0$,
then $A(t)$ gets replaced by $A(t+t_0)-A(t_0)$.
\end{remark}
This observation leads to the following theorem.
%
\begin{theorem}\label{indep}
Let $\cX$ be a complete separable metric space and for each $N \ge1$
let $F_N(\omega)$ be a $\cX$-valued and ${\cG}_{N,k_N}$-measurable random
variable. Suppose that the distribution
$\lambda_N$ of $F_N$ under $P_N$ converges weakly as $N\to\infty$
to $\lambda$ on $\cX$ and $\frac{k_N}{N}\to t_0$. Let the conditions of
Theorem \ref{TriCLT} hold true and set
\[
\xi_{N,k_N}(t)=\frac{1}{\sqrt N}\sum_{k_N+1\le n\le k_N+Nt}
U_{N,n}.
\]
Then the joint distribution of the pair $(F_N, \xi_{N,k_N}(\cdot))$
converges on $\cX\times D[0,T]$ to the product of $\lambda$ and the
distribution ${\gamma}$ of a
Gaussian process with independent increments having mean $0$
and variance $A(t+t_0)-A(t_0)$. In particular, any limit in
distribution of
\[
\xi_N(t)=\frac{1}{\sqrt N} \sum_{1\le n\le Nt}
U_{N,n}
\]
is always a process with independent increments. We can drop the
assumption that $\frac{k_N}{N}\to t_0$ provided we can verify that
for some $A(t)$,
\[
\lim_{N\to\infty}\biggl\|\frac{1}{ N}\sum_{k_N+1\le n\le k_N+Nt}
W^2_{N,n}- A(t)\biggr\|_{L_1(P_N)}=0.
\]
\end{theorem}
\begin{pf}
Since the conditions of Theorem \ref{TriCLT}
are satisfied here, $\xi_{N,k_N}$
converges in distribution as $N\to\infty$ to a Gaussian process with
independent increments whose distribution we denote by ${\gamma}$.
Now, if
$\mu_N$ denotes the joint distribution of $F_N$ and $\xi_{N,
k_N}(\cdot)$, the convergence of the marginals implies the tightness
of $\mu_N$. Taking a subsequence if necessary, we can assume that
$\mu_N$ has a limit $\mu$ with marginals $\lambda$ and $\gamma$. We need
to prove that $\mu=\lambda\times\gamma$.
It is enough to prove that if $E\subset\cX$ and $F\subset D[0,T]$
are continuity sets of $\lambda$ and $\gamma$, respectively, then
$\mu(E\times F)=\lambda(E)\times\gamma(F)$.
We can assume without loss of generality that $\lambda(E)>0$. Set
$B_N=\{\omega\dvtx  F_N(\omega)\in E\}$, then $P_N(B_N)\to
\lambda(E)$, and so $P_N(B_N)\ge\frac{1}{2}\lambda(E)> 0$ for $N$ large
enough. In view of Remark \ref{conditional}, $\xi_{N, k_N}(\cdot)$ converges
in distribution under $P_{N,B_N}$ as $N\to\infty$ to a Gaussian process
with independent increments and since, clearly, under $P_{N,B_N}$ we have
convergence in B4 to the same $\tilde A(t)=A(t+t_0)-A(t_0)$
as under $P_N$, it follows that the distribution of $\xi_{N,
k_N}(\cdot)$
under $P_{N,B_N}$ converges to $\gamma$. In particular, since $F$
is a continuity set,
\[
P_{N,B_N}\bigl\{ \omega\dvtx  \xi_{N,k_N}(\cdot)\in F\bigr\}=
\frac{\mu_N(E\times F)} {
P_N(B_N)}\to\gamma(F).
\]
Since $E\times F$ is a continuity set of $\mu$, this proves that
$\frac{\mu(E\times F)}{\lambda(E)}=\gamma(F)$.
\end{pf}
%
\begin{corollary}\label{vectorCLT}
Assume that we have a triangular array consisting of
${\cG}_{N,n}$-measurable random vectors $U_{N,n}\dvtx \Omega\to R^d$ and
that each linear combination $\langle\lambda, U_{N,n}\rangle$ satisfies
the assumptions \textup{B1--B4}. In particular,
\[
\lim_{N\to\infty} \biggl\|\biggl[\frac{1}{N}\sum
_{1\le n\le Nt} \langle \lambda, W_{N,n}\rangle^2
\biggr]-\bigl\langle\lambda, A(t)\lambda\bigr\rangle\biggr\|_{L_1(P)}=0.
\]
Then
\[
\xi_N(t)=\frac{1}{\sqrt N}\sum_{k_N+1\le n\le k_N+Nt}
U_{N, n}
\]
converges in distribution on the Skorokhod space $D[[0,T]; R^d]$
to the Gaussian process $\eta(t)$ with independent increments taking
values in $R^d$, having mean $0$ and covariance
\[
E\bigl[\bigl\langle\lambda\bigl(\eta(t)-\eta(s)\bigr)\bigr\rangle^2
\bigr]=\bigl\langle\lambda, \bigl(A(t)-A(s)\bigr) \lambda\bigr\rangle.
\]
\end{corollary}
\begin{pf}
By the results for the scalar case, the distribution of $\langle u,
\xi_N(t)\rangle$ converges to a Gaussian process with independent
increments. This implies compactness of the distributions of the
vector process $\xi_N(\cdot)$. Let $Q$ be a limit point of distributions
of $\xi_N$ and let $\eta$ be the corresponding limiting vector process.
By the above for each constant vector $u$, the distribution of the increments
$\langle u, \eta(t)-\eta(s)\rangle$ must be Gaussian and, therefore,
by the
Cram\'er--Wold argument, $\eta(t)-\eta(s)$ has under $Q$ the
$d$-dimensional Gaussian distribution with mean $0$ and a covariance
matrix $\{A_{i,j}(t)-A_{i,j}(s)\}$. Moreover, by Theorem \ref{indep},
under $Q$ the random variable $\langle u,\eta(t)-\eta(s)\rangle$ is
independent of $\{\eta(\tau)\dvtx \tau\le s\}$ for every $t>s$ and $u\in R^d$.
This is sufficient to determine $Q$ as the distribution of a
Gaussian process $\eta(t)$ with independent increments taking
values in $R^d$ having mean $0$ and covariance
\[
E\bigl[\bigl(\eta_i(t)-\eta_i(s)\bigr) \bigl(
\eta_j(t)-\eta_j(s)\bigr)\bigr]=A_{i,j}(t)-A_{i,j}(s)
\]
and to establish that the distribution of
\[
\xi_N(t)=\frac{1}{\sqrt N}\sum_{k_N+1 \le n\le k_N+Nt}
U_{N,n}
\]
converges to $Q$ on the Skorokhod space $D[[0,T]; R^d]$.
\end{pf}

Next, we break the proof of Theorem \ref{MainThm} into several steps and
use the following representations:
%
\begin{eqnarray}
\label{repr}\qquad Y_{i,q_i(n)}&=&Y_{i,q_i(n),1}+\sum
_{r=1}^\infty[Y_{i,q_i(n), 2^r}- Y_{i,q_i(n),2^{r-1}}],
\nonumber\\
\zeta_{i,N,0}(t)&=&\frac{1}{\sqrt N}\sum_{1\le n\le M_i(Nt)}
Y_{i,q_i(n),1},
\nonumber\\[-8pt]\\[-8pt]
\zeta_{i,N,r}(t)&=&\frac{1}{\sqrt N}\sum_{1\le n\le M_i(Nt)}[Y_{i,q_i(n),2^r}-
Y_{i,q_i(n),2^{r-1}}],\qquad r\ge1,\quad\mbox{and}
\nonumber
\\
\xi_{i,N}(t)&=&\sum_{r=1}^\infty
\zeta_{i,N,r}(t),
\nonumber
\end{eqnarray}
where $M_i(u)=u$ if $i\geq k+1$ and $M_i(u)=u/i$ for $i=1,\ldots,i$.
First, we establish the following.
%
\begin{proposition} \label{subThm1}
For each fixed $u$, as $N$ goes to $\infty$, the partial sums
\[
\xi^u_{i,N}(t)=\sum_{r=1}^u
\zeta_{i,N,r}(t)=\sum_{1\le n\le M_i(Nt)}
Y_{i,q_i(n),2^u}
\]
form a tight family of processes on the Skorokhod space $D[[0,t];R^k]$.
All the
limit points are Gaussian processes with independent increments. The second
moments are uniformly integrable so that the covariance of the limiting
Gaussian process can be identified as the limit of the covariances of the
corresponding approximating processes along the subsequence.
\end{proposition}
\begin{pf} We note that $Y_{i,q_i(n),r}$
is $\cF_{-\infty,q_i(n)+r}$ measurable. In order to apply Theorem
\ref{TriCLT}
with $\cG_{N,n}=\cF_{-\infty, q_i(n)+r}$, we need to verify the conditions
B1--B4. With such choice of $\cG_{N,n}$, B1 is clearly
fulfilled. To verify the uniform square integrability
of $\{Y_{i,q_i(n),r}\}$, we observe that the uniform square
integrability of
any family $\{Z_\alpha\}$ implies the uniform integrability of
$\{ E[Z_\alpha|\cG]\}$ as $\alpha$ and $\cG$ vary. The distribution
of $\{X(n)\}$ is the same for all $n$ and, therefore, by our moment
condition, $|X(n)|^{2\iota}$ are uniformly integrable. Using the bound
$|F|\le C(1+\sum|x_i|^\iota)$, it is easily seen that $\{
Y_{i,q_i(n),r}\}$
are uniformly square integrable. To control
$\|E[Y_{i,q_i(n),r}|\cF_{-\infty, l}]\|_{2}$, we use Corollary \ref{cor}(ii)
for $q_{i-1}(n)+r\le l$, which yields the estimate
\[
\bigl\|E[Y_{i,q_i(n),r}|\cF_{-\infty, l}]\bigr\|_2\le c(d,p,\ka,\iota)c(
\gamma_m,\gamma_{ q\iota}) \varpi_{q,p}
\bigl(q_i(n)-r-l\bigr)
\]
provided $q_i(n)\ge l+r$. On the other hand, if $q_{i-1}(n)+r\ge l$, we can
write
\begin{eqnarray*}
\bigl\|E[Y_{i,q_i(n),r}|\cF_{-\infty, l}]\bigr\|_2&\le&
\bigl\|E[Y_{i,q_i(n),r}|\cF_{-\infty,
q_{i-1}(n)+r}]\bigr\|_2
\\
&\le& c(d,p,\ka,\iota) c(\gamma_m,\gamma_{ q\iota})
\varpi_{q,p}\bigl(q_i(n)- q_{i-1}(n)-2r\bigr)
\\
&\le& c(d,p,\ka,\iota) c(\gamma_m,\gamma_{ q\iota})
\varpi_{q,p}(n-2r),
\end{eqnarray*}
whenever $n\geq2r$ and $n\ge n^*=n^*(i)=\min\{ m\dvtx
q_i(l)-q_{i-1}(l)\geq l
$ $\forall l\geq m\}$, observing that $n^*<\infty$ by (\ref{2q3}).
Assuming that $q\ge p$, we can always bound $\varpi_{p,q}$ by $1$.
Therefore, choosing $c(n)=1$ for small values of $n$
(there are at most $n^*+2r$ of them) and estimating $c(n)$ by either
$c(d,p,\ka,\iota) c(\gamma_m,\break \gamma_{ q\iota}) \varpi_{q,p}(q_i(n)-r-l)$ or
by $c(d,p,\ka,\iota) c(\gamma_m,\gamma_{ q\iota}) \varpi_{q,p}(n-2r)$,
we arrive at B3 with the estimate
\[
\sum_{n=0}^\infty c(k) \le\Biggl[n^*+2r+2
\sum_{n=1}^\infty\varpi_{p,q}(n)
\Biggr] c(d,p,\ka,\iota) c(\gamma_m,\gamma_{ q\iota}).
\]
If we set
\[
R_{i, m,r}=\sum_{n\ge m-r}E[Y_{i,n,r}|
\cF_{-\infty,m}],
\]
then it follows from the above estimates that
%
\begin{equation}
\label{Rest} \sup_{i,l}\|R_{i,l,r}\|_2\le2
\bigl(n^*+r+\theta(p,q)\bigr)c(d,p,\ka,\iota) c(\gamma_m,
\gamma_{ q\iota}),
\end{equation}
where $\theta(p,q)$ is given by (\ref{eqsix1}).
It is now clear that $W_{i,n,r}= Y_{i,n-r,r}+R_{i,n+1,r}-R_{i,n,r}$ is
a martingale difference and is uniformly square integrable. While
B4 may not hold, the limit will exist along suitable subsequences.
The uniform bound on $\|W_{i,n,r}\|_2 $ ensures that limits $A(t)$ will
be Lipschitz continuous functions of $t$ and the convergence is uniform
in $t$.
\end{pf}

In order to obtain convergence of processes $\xi_{i,N}$ and not only their
approximations $\xi_{i,N,r}$, we will need uniform bounds in the
representations (\ref{repr}).
%
\begin{proposition} \label{subThm2} The differences $\{\zeta_{i,N,r}(t)\}$
satisfy
%
\begin{equation}
\label{zetaest} \sum_r \sup_{N\ge1}
\max_{1\le i\le\ell} \Bigl\|\sup_{0\le t\le
T}\bigl|\zeta_{i,N,r}(t)\bigr|
\Bigr\|_2\le C<\infty.
\end{equation}
\end{proposition}
\begin{pf} Set $\tilde Y_{i,n,r}=Y_{i,n,2^r}-Y_{i,n,2^{r-1}},  r\geq1$
and
\[
\tilde R_{i,n,r}=\sum_{m\geq n+1}E(\tilde
Y_{i,m,r}|\cF_{-\infty,n+2^r}).
\]
Estimating conditional expectations here by Corollary \ref{cor}(iv)
when $m-n\geq
2^{r+1}$ and by the contraction argument when $n+1\leq m\leq
n+2^{r+1}$, and
applying Corollary \ref{cor}(iv) after that again, we obtain
%
\begin{eqnarray}
\label{zeta-1} \|\tilde R_{i,n,r}\|_2&\leq&2^{r+1}
\sup_n\|\tilde Y_{i,n,r}\|_2+\tilde C \bigl(
\bigl(\be\bigl(q,2^r\bigr)\bigr)^\del+\bigl(\be
\bigl(q,2^{r-1}\bigr)\bigr)^\del \bigr)
\nonumber\\[-8pt]\\[-8pt]
&\leq&\hat C2^r \bigl(\bigl(\be\bigl(q,2^r\bigr)
\bigr)^\del+\bigl(\be\bigl(q,2^{r-1}\bigr)
\bigr)^\del \bigr),
\nonumber
\end{eqnarray}
where $\tilde C,\hat C>0$ do not depend on $i,n,r$. Now observe that
%
\begin{equation}
\label{zeta-2}\qquad \zeta_{i,N,r}=\frac1{\sqrt N}\sum
_{1\leq m\leq M_i(NT)}Z_{i,q_i(m),r} -\frac1{\sqrt N} (\tilde
R_{i,q_i([M_i(NT)]),r}-\tilde R_{i,0,r} ),
\end{equation}
where $Z_{i,n,r}=\tilde Y_{i,n,r}+\tilde R_{i,n,r}-\tilde R_{i,n-1,r},
n\geq1$ is a martingale differences sequence with respect to the filtration
$\{\cG_n,  n\geq1\}$ with $\cG_n=\cF_{-\infty,n+2^r}$. By the Doob
inequality for martingales,
%
\begin{eqnarray}
\label{zeta-3}
\frac1NE\sup_{0\leq t\leq T} \biggl\vert\sum
_{1\leq l\leq Nt}Z_{i,q_i(l),r} \biggr\vert^2&\leq&\frac4N\sum
_{1\leq l\leq NT}EZ^2_{i,q_i(l),r}
\nonumber\\
&\leq&4T\max_{1\leq l\leq NT}EZ^2_{i,q_i(l),r}\\
&\leq&12T \Bigl(
\sup_n \|\tilde Y_{i,n,r}\|_2+2\sup_n
\|\tilde R_{i,n,r}\|_2 \Bigr).
\nonumber
\end{eqnarray}
We can estimate also
%
\begin{eqnarray}
\label{zeta-4}
\frac1NE\max_{0\leq l\leq NT}|\tilde R_{i,q_i(l),r}-\tilde
R_{i,0,r}|^2
&\leq&\frac4N\sum_{1\leq l\leq NT}E\tilde R^2_{i,q_i(l),r}
\nonumber\\[-8pt]\\[-8pt]
&\leq&4\max_{0\leq l
\leq NT}R\tilde R^2_{i,q_i(l),r}.
\nonumber
\end{eqnarray}
Now collecting (\ref{zeta-1})--(\ref{zeta-4}) and applying Corollary
\ref{cor}(iv) again to (\ref{zeta-3}) and (\ref{zeta-4}), we obtain that
%
\begin{equation}
\label{zeta-5} \sup_{N\geq1}\Bigl\|\sup_{0\leq t\leq T}\bigl|\zeta_{i,N,r}(t)\bigr|
\Bigr\|_2\leq \hspace*{1.5pt}\tilde{\hspace*{-1.5pt}\tilde C} 2^r \bigl(\bigl(\be
\bigl(q,2^r\bigr)\bigr)^\del+\bigl(\be\bigl(q,2^{r-1}
\bigr)\bigr)^\del \bigr),
\end{equation}
where $\hspace*{1.5pt}\tilde{\hspace*{-1.5pt}\tilde C}>0$ does not depend on $r$. Since
$\sum_{r\geq1}(\be(q,r))^\del$ converges by our assumption (\ref{beta1}),
then $\sum_{r\geq1}2^r(\be(q,2^r))^\del$ converges as well, and so the
right-hand side of (\ref{zeta-5}) is summable, implying (\ref{zetaest}).
\end{pf}

Next, we deal specifically with the terms $Y_{i,q_i(n)}$,
$ k+1\le i\le\ell$ which satisfy (\ref{2q1}), (\ref{2q2}) and
(\ref{2q3}).
By Propositions \ref{subThm1} and \ref{subThm2}, any possible limit
$\eta_i(t)$ in distribution of
\[
\xi_{i,N}(t)=\frac{1}{\sqrt N} \sum_{n\le Nt}
Y_{i,q_i(n)}
\]
for $1\le i\le\ell$ will be a Gaussian process with independent increments.
The processes $\{\eta_i(\cdot),  k+1\le i\le\ell\}$ will be mutually
independent as well as totally independent of $\{\eta_i(\cdot),
1\le i\le k\}$, which is proved by successive application of Theorem
\ref{indep}. We note that it is enough to show that for any $T<\infty$
we can ignore $\sum_{n\leq k_N(i)} Y_{i,q_i(n)}$ in the definition of
$\xi_{i,N}(t)$, where $k_N(i)=\max\{ n\dvtx   q_i(n)\le q_{i-1}(NT)\}$ so that
Theorem \ref{indep} will be applicable then to the approximations
\[
\xi_{i,N,r}(t)=\frac{1}{\sqrt N} \sum_{k_N(i)+1\leq n\le Nt}
Y_{i,q_i(n),r}
\]
with $Y_{i,q_i(n),r}$ defined at the beginning of Section \ref{sec4}.
At the end, relying on Proposition \ref{subThm2}, we can let $r\to
\infty$ and complete the proof. From (\ref{2q3}), for any $\epsilon>0$,
$q_i(N\epsilon)\ge q_{i-1}(NT)$ for large $N$, which implies that the
initial terms are at most $N\epsilon$ in number. Since $\epsilon$ is
arbitrary, we see that $N^{-1}k_N(i)\to0$ as $N\to\infty$. By
(\ref{varbound}) of Lemma \ref{lem52}, we obtain that the contribution
of initial $k_N(i)$ terms in the sum for $\xi_{i,N}$ is negligible.
Similarly, we conclude that it does not matter whether we take the sum
for $\xi_{i,N,r}(t)$ above until $Nt$ or until $Nt+k_N(i)$ as in
Theorem~\ref{indep}. By Proposition \ref{covariance2}, we have also
that the limiting variance $A_{i,i}(t)$ of each $\xi_{i,N}(t),  i>k$,
exists and is given by (\ref{cov,ii}).

We observe that independency of processes $\eta_i,
i>k$, of each other and of $\eta_i,  i\leq k$, can be proved in an
alternative way without using Theorem \ref{indep}. Namely, we can rely on
Theorem \ref{TriCLT} showing that linear combinations of processes
$\xi_{i,N,r}$ converge to Gaussian processes, deriving similarly to
the above
via uniform estimates of Proposition \ref{subThm2} that linear combination
of processes $\eta_i$ are Gaussian and concluding the proof via the vanishing
covariances assertion (\ref{cov,ij}) of Proposition~\ref{covariance2}.

Now, we are able to complete the proof of Theorem \ref{MainThm}. First,
we conclude from Propositions \ref{subThm1} and \ref{subThm2}
together with
Corollary \ref{vectorCLT} that the $k$-dimensional process $\{\xi_{i,N}(t)\dvtx
  1\leq i\leq k\}$ converges in distribution as $N\to\infty$ to a Gaussian
process $\{\eta_i(t)\dvtx   1\leq i\leq k\}$ with stationary independent
increments whose covariances are given by Proposition \ref
{covariance}. As
explained above, when $i\geq k+1$, the process $\xi_{i,N}(t)$
converges in
distribution to a Gaussian process $\eta_i(t)$ with stationary independent
increments and $\eta_{k+1}(t),\ldots,\eta_\ell(t)$ are both mutually independent
and independent of processes $\eta_1(t),\ldots,\eta_k(t)$. It follows
that the
$\ell$-dimensional process $\{\xi_{i,N}(t)\dvtx   1\leq i\leq\ell\}$ converges
in distribution as $N\to\infty$ to the Gaussian process $\{\eta_i(t)\dvtx
1\leq i\leq\ell\}$ with stationary independent increments whose covariances
are given by Propositions \ref{covariance} and \ref{covariance2}
taking into account independency of processes $\eta_i(t)$ with $i\geq k+1$
of other processes $\eta_j(t)$ with $j\ne i$.

It remains to show that the process $\xi_N(t)$ given by (\ref{xi}) converges
in distribution as $N\to\infty$ to a Gaussian process $\xi(t)$ given by
(\ref{eqmain1}). The convergence itself is clear since each $\xi_{i,N}$
converges to the corresponding $\eta_i$. In order to show that $\xi$ is
a Gaussian process, it suffices to prove the same for $\zeta(t)=\sum_{i=1}^k
\eta_i(it)$ since $\tilde\zeta(t)=\sum_{i=k+1}^\ell\eta_i(t)$ is
a Gaussian
process (as a sum of independent Gaussian processes) independent of
$\zeta$,
and so $\zeta(t)+\tilde\zeta(t)$ is a Gaussian process if $\zeta
(t)$ is.
Since $(\eta_1(t),\ldots,\eta_k(t))$ is a $k$-dimensional Gaussian process
with independent increments, then the vector increments
$ (\eta_j(it)-\eta_j((i-1)t),  j=1,2,\ldots,k )$ for
$i=1,2,\ldots,k$ are
mutually independent $k$-dimensional Gaussian processes, and so
\[
\zeta_\la(t)=\sum_{i=1}^k
\sum_{j=1}^k\la_{ij} \bigl(
\eta_j(it)-\eta_j\bigl((i-1)t\bigr) \bigr)=\sum
_{j=1}^k\sum_{i=1}^k
\la_{ij} \bigl(\eta_j(it)-\eta_j
\bigl((i-1)t\bigr) \bigr)
\]
is a Gaussian process for any choice of constants $\la_{ij}$ and we recall
that $\eta_j(0)=\xi_{j,N}(0)=0$. Now observe that choosing $\la_{ij}=1$ if
$i\leq j$ and $\la_{ij}=0$ otherwise, we obtain that $\zeta_\la
(t)=\zeta(t)$,
completing the proof.

As to our claim that increments of $\xi(t)$ may not be independent if
$k\geq2$,
consider, for instance, the case $k=\ell=2$ and
\[
\xi(t)-\xi(t/2)=\eta_1(t)+\eta_2(2t)-
\eta_1(t/2)-\eta_2(t)   \quad\mbox{and}\quad   \xi(t/2)=
\eta_1(t/2)+\eta_2(t).
\]
Then by Proposition \ref{covariance},
\[
E \bigl(\xi(t/2) \bigl(\xi(t)-\xi(t/2) \bigr)\bigr)=D_{2,1}t/2,
\]
where
\[
D_{2,1}=\frac12\sum^\infty_{u=-\infty}a_{2,1}(u)
\]
and
\[
a_{2,1}(u)=\int F_2(x,y)F_1(z)\,d\mu(x)\,d
\mu_u(y,z).
\]
Assume, for instance, that $X(0),X(1),X(2),\ldots$ is a sequence of independent
identically distributed random variables, then $\mu_u=\mu\times\mu$
if $u\ne0$,
and so $a_{2,1}(u)=0$ if $u\ne0$, while
\[
a_{2,1}(0)=\int F_2(x,y)F_1(y)\,d\mu(x)\,d\mu(y).
\]
Now suppose that $EX(0)=0$, $EX^2(0)=1$ and choose $F(x,y)=x^2y^2-1$. Then
$\int F(x,y)\,d\mu(x)\,d\mu(y)=0$, $F_2(x,y)=x^2(y^2-1)$, $F_1(x)=x^2-1$,
and so
\[
D_{2,1}=\frac12 a_{2,1}(0)=\int\bigl(y^2-1
\bigr)^2\,d\mu(y)\ne 0
\]
unless $X^2(0)=1$ with probability one.

\section{Continuous time case}\label{sec7}

First, we represent again the function $F$ in the form (\ref{F}) and
$\xi_N(t)$ given by (\ref{cont}) in the form (\ref{xi}) where now
%
\begin{equation}
\label{81} \xi_{i,N}(t)=\frac1{\sqrt N}\int_0^{S_i(Nt)}F_i
\bigl(X\bigl(q_1(s)\bigr),\ldots,X\bigl(q_i(s)\bigr)
\bigr)\,ds
\end{equation}
with $S_i(u)=u/i$ if $i\leq k$ and $S_i(u)=u$ if $i\geq k+1$.
Set
\begin{eqnarray*}
F_{i,r,t}&=&F_{i,r,t}(x_1,\ldots,x_{i-1},{
\omega})=E \bigl(F_i\bigl(x_1,\ldots,x_{i-1},X(t)\bigr)| \cF_{t-r,t+r}
\bigr),
\\
X_r(t)&=&E\bigl(X(t)|\cF_{t-r,t+r}\bigr),\\
Y_i(t)&=&F_i \bigl(X\bigl(q_1(s)\bigr),\ldots,X
\bigl(q_i(s)\bigr) \bigr)   \qquad\mbox{if }   t=q_i(s)
\end{eqnarray*}
and
\begin{eqnarray*}
Y_i(t)&=&0  \qquad\mbox{if }   t\ne q_i(s)
\qquad \mbox{for any }   s,\\
Y_{i,r}(t)&=&F_{i,r,t}
\bigl(X_r\bigl(q_1(s)\bigr),\ldots,X_r
\bigl(q_i(s)\bigr) \bigr)\qquad\mbox{if } t=q_i(s)
\end{eqnarray*}
and
\begin{eqnarray*}
Y_{i,r}(t)=0\qquad\mbox{if } t\ne q_i(s)  \qquad\mbox{for any } s.
\end{eqnarray*}

In order to use fully our discrete time technique, it will be convenient
to pass from $\xi_{i,N}$ to $\tilde\xi_{i,N}$ given by
\[
\tilde\xi_{i,N}(t)=\frac1{\sqrt N}\sum_{n=0}^{[S_i(Nt)]}I_i(n),
\]
where $I_i(n)=\int_n^{n+1}Y_i(q_i(s))\,ds$. The error of such transition is
estimated by
%
\begin{equation}
\label{82} \sup_{0\leq t\leq T}\bigl|\xi_{i,N}(t)-\tilde
\xi_{i,N}(t)\bigr|\leq\frac 1{\sqrt N} \max_{0\leq n\leq NT}Q_i(n),
\end{equation}
where $Q_i(n)=\int_0^1|Y_i(q_i(n+s))|\,ds$. Now for any $\del>0$,
\begin{eqnarray*}
P\Bigl\{\max_{0\leq n\leq NT}Q_i(n)>\ve\sqrt N\Bigr\}&\leq& NT
\max_{0\leq n\leq NT}P\bigl\{ Q_i(n)>\ve\sqrt N\bigr\}
\\
&\leq&\frac T{\ve^2}\max_{0\leq n\leq NT}\int_{\{ Q_i(n)>\ve\sqrt
N\}}Q^2_i(n)
\,dP
\\
&\leq&(\ve\sqrt N)^{-\del}\int Q^{2+\del}_i(n)\,dP\\
&\leq&(\ve
\sqrt N)^{-\del} \int_0^1
EY^{2+\del}_i\bigl(q_i(n+s)\bigr)\,ds\\
&\leq& C(\ve\sqrt
N)^{-\del}.
\end{eqnarray*}
Thus, the left-hand side of (\ref{82}) tends to 0 in probability as
$N\to
\infty$, and so it suffices to prove our functional central limit theorem
for $\tilde\xi_{i,N}$ in place of $\xi_{i,N}$.

Introduce the approximations $\tilde\xi_{i,N,r}$ of $\tilde\xi_{i,N}$ by
%
\begin{equation}
\label{83} \tilde\xi_{i,N,r}(t)=\frac1{\sqrt N}\sum
_{n=0}^{[S_i(Nt)]}I_{i,r}(n),
\end{equation}
where $I_{i,r}(n)=\int_n^{n+1}Y_{i,r}(q_i(s))\,ds$. Now set
\[
R_{i,r}(m)=\sum_{l=m+1}^\infty E
\bigl(I_{i,r}(l)|\cF_{-\infty,m+r} \bigr)
\]
and $Z_{i,r}(m)=I_{i,r}(m)+R_{i,r}(m)-R_{i,r}(m-1)$. Then
$E(Z_{i,r}(m)|\cF_{-\infty,m-1+r})=0$, and so $\{ Z_m, \cG_m\}_{m\geq0}$
with $Z_m=Z_{i,r}(m)$ and $\cG_m=\cF_{-\infty,m+r}$ turns out to be a
martingale differences sequence.
We saw already above that $\{ Q_i^2(n)\}$ is uniformly integrable. Then
both $\{ I^2_i(n)\}$ and $\{ I^2_{i,r}(n)\}$ are uniformly integrable and,
like in the proof of Proposition \ref{subThm1}, we conclude that both
$\{ R^2_{i,r}(n)\}$ and $\{ Z^2_{i,r}(n)\}$ are uniformly integrable as
well. Set
\[
\zeta_{i,N,r}(t)=\frac1{\sqrt N}\sum^{[ S_i(Nt)]}_{n=0}Z_{i,r}(n).
\]
Then, similar to Section \ref{sec5}, we obtain that
%
\begin{equation}
\label{84} \sup_{0\leq t\leq T}\bigl|\tilde\xi_{i,N,r}(t)-
\zeta_{i,N,r}(t)\bigr|\to0 \qquad\mbox{in probability as }   N\to\infty
\end{equation}
and so in order to obtain a central limit theorem for $\tilde\xi_{i,r,N}(t)$,
it suffices to prove it for the normalized martingal $\zeta_{i,r,N}(t)$.

In order to invoke martingale limit theorems,we have to study next the
asymptotical behavior as $N\to\infty$ of normalized variances
$E ( \zeta_{i,r,N}(S_i(Nt)) )^2$. As in the discrete time case
considered in Section \ref{sec4}, in view of (\ref{F}) and (\ref{81}),
it suffices to study the asymptotical behavior of
%
\begin{eqnarray}
\label{85} D_{i,j}(N,s,t)&=&E\bigl[\xi_{i,N}(s)
\xi_{j,N}(t)\bigr]
\nonumber\\[-8pt]\\[-8pt]
&=&\frac1N\int_0^{S_j(Nt)} \int_0^{S_i(Ns)}E
\bigl[Y_i\bigl(q_i(u)\bigr)Y_j
\bigl(q_j(v)\bigr)\bigr]\,du\,dv.
\nonumber
\end{eqnarray}
We treat first the case when $1\leq i,j\leq k$ similarly to Proposition
\ref{covariance}. Let ${\upsilon}$ be the greatest common divisor of
$i$ and $j$,
then, similarly to the argument in Lemma \ref{Lb2}, we obtain that for any
integer $w$,
%
\begin{equation}
\label{86} \lim_{u,v\to\infty,  iu-jv=w{\upsilon
}}E\bigl[Y_i(iu)Y_j(jv)
\bigr]=a_{i,j}(w,2w,\ldots,{\upsilon}w)
\end{equation}
with $a_{i,j}$ defined in Proposition \ref{covariance}. Now, changing
variables, we have
%
\begin{eqnarray}
\label{87} &&\frac1N\int_0^{Nt/j}\int
_0^{Ns/i}E\bigl[Y_i(iu)Y_j(jv)
\bigr]\,du\,dv
\nonumber\\[-8pt]\\[-8pt]
&&\qquad=\frac{\upsilon}{Ni} \int_0^{Nt/j}\int
_{-jv/{\upsilon}}^{(Ns-jv)/{\upsilon}}E\biggl[Y_i \biggl(
\frac{jv+w{\upsilon}}i \biggr) Y_j(jv)\biggr]\,dw\,dv.
\nonumber
\end{eqnarray}
When $v$ is large, then the expectation under the integral equals approximately
$a_{i,j}(w,2w,\ldots,{\upsilon}w)$ and taking into account that the
latter is absolutely
integrable in $w$ from $-\infty$ to $\infty$, we can approximate the interior
integral in $w$ by the integral $\int_{-\infty}^\infty$. Next we
integrate in
$v$ within constraints $0\leq v\leq Nt/j$ and $u=(jv+w{\upsilon
})/i\leq Ns/i$, that is,
asymptotically for $N$ large $0\leq v\leq\frac Nj\min(s,t)$. It
follows that
the expression in (\ref{87}) is approximately equal as $N\to\infty$ to
%
\begin{equation}
\label{88} \frac{{\upsilon}}{ij}\min(s,t)\int_{-\infty}^\infty
a_{i,j}(w,2w,\ldots,{\upsilon}w)\,dw
\end{equation}
and we obtain the same covariances as in the discrete time case.

Next, we claim that for each $i=k+1,\ldots,\ell$ and $t>0$,
%
\begin{equation}
\label{89} \lim_{N\to\infty}D_{i,i}(N,t,t)=0.
\end{equation}
Indeed, set again $b_{i,j}(u,v)=E ( Y_i(q_i(u))Y_j(q_j(v) )$. Then
%
\begin{eqnarray}
\label{810}
\frac1N\int_0^{Nt}\int
_0^{Nt}\bigl|b_{i,i}(u,v)\bigr|\,du\,dv&\leq&\frac2N
\int_0^{Nt}du \int_u^{u+{\gamma}}\bigl|b_{i,i}(u,v)\bigr|\,du\,dv
\nonumber\\
&&{}+\frac2N\int_0^{N{\gamma}}du\int
_{u+{\gamma}}^{Nt}\bigl|b_{i,i}(u,v)\bigr|\,du\,dv\nonumber\\[-8pt]\\[-8pt]
&&{}+ \frac2N\int
_{N{\gamma}}^{Nt}du\int_{u+{\gamma
}}^{Nt}\bigl|b_{i,i}(u,v)\bigr|\,du\,dv
\nonumber
\\
&\leq& C\bigl(t{\gamma}+{\gamma}+t\be^{(i)}_{\gamma}(N{\gamma })
\bigr)
\nonumber
\end{eqnarray}
for some $C>0$ independent of $t,  N$ and ${\gamma}$, where we obtain by
(\ref{new2q}) and estimates similar to Lemma \ref{lem52} and Proposition
\ref{covariance2} that for any $i>k$ and ${\gamma}>0$,
%
\begin{equation}
\label{811}\qquad \be_{\gamma}^{(i)}(M)=\sup_{u\geq M}\int
_{u+{\gamma}}^\infty \bigl|b_{i,i}(u,v)\bigr|\,dv<\infty
\quad\mbox{and}\quad  \lim_{M\to\infty}\be_{\gamma}^{(i)}(M)=0.
\end{equation}
So, letting first $N\to\infty$ and then ${\gamma}\to0$, we obtain
(\ref{89}).
%
\begin{remark}\label{rem81}
In fact, in the continuous time case we can take $q_i(t)=\al_it$ for arbitrary
$0<\al_1<\al_2<\cdots<\al_k$ in place of $1<2<\cdots<k$ while leaving
$q_i(t),  i=k+1,\ldots,\ell$ as before. In this situation (\ref{86}) becomes
\[
\lim_{u,v\to\infty,  \al_iu-\al_jv=z}E\bigl[Y_i(\al_iu)Y_j(
\al_jv)\bigr]= a_{i,j}(\rho_1z,
\rho_2z,\ldots,\rho_{n_{ij}}z,z),
\]
where $\rho_1<\rho_2<\cdots<\rho_{n_{ij}}<1$ and $\al_i\rho_l,\al_j\rho_l\in
\{\al_1,\ldots,\al_k\}$ for $l=1,\ldots,n_{ij}$. Then the covariances (\ref{88})
will have the form
\[
\frac{1}{\al_i\al_j}\min(s,t)\int_{-\infty}^\infty
a_{i,j}(\rho_1w,\rho_2w,\ldots,
\rho_{n_{ij}}w,w)\,dw.
\]
\end{remark}




\printaddresses

\end{document}